\documentclass[12pt]{amsart}
\usepackage{amsmath,amsthm,amsfonts,amssymb,xcolor}
\usepackage[all]{xy}
\usepackage{graphicx}
\usepackage{mathrsfs}
\usepackage{enumitem}
\usepackage{etoolbox}
\apptocmd{\sloppy}{\hbadness 10000\relax}{}{}
\apptocmd{\sloppy}{\vbadness 10000\relax}{}{}
\usepackage{hyperref}
\usepackage[letterpaper,margin=1.1in]{geometry}

\numberwithin{equation}{section}
\theoremstyle{plain}
\newtheorem{thm}{Theorem}[section]
\newtheorem{prop}[thm]{Proposition}
\newtheorem{cor}[thm]{Corollary}

\newtheorem{lem}[thm]{Lemma}

\theoremstyle{definition}
\newtheorem{rem}[thm]{Remark}

\newtheorem{ex}[thm]{Example}

\def\R{\mathbb{R}}
\def\Z{\mathbb{Z}}
\def\N{\mathbb{N}}

\def\g{\gamma}
\def\d{\delta}

\def\G{\Gamma}
\def\a{\alpha}

\newcommand{\diam}{\mathop\mathrm{diam}\nolimits}
\newcommand{\dist}{\mathop\mathrm{dist}\nolimits}

\newcommand{\Hold}{\mathop\mathrm{H\ddot{o}ld}\nolimits}
\newcommand{\Lip}{\mathop\mathrm{Lip}\nolimits}

\newcommand{\sdim}{\operatorname{s-dim}}
\newcommand{\card}{\operatorname{card}}
\newcommand{\val}{\operatorname{val}}

\newcommand{\Haus}{\mathcal{H}}

\newcommand{\interior}{\mathop\mathrm{int}\nolimits}

\begin{document}

\title[H\"older parameterization of IFS]{H\"older parameterization of iterated function systems and a self-affine phenomenon}

\author{Matthew Badger \and Vyron Vellis}
\thanks{M.~Badger was partially supported by NSF DMS grant 1650546. V.~Vellis was partially supported by NSF DMS grants 1800731 and 1952510.}
\date{November 1, 2020}

\address{Department of Mathematics\\ University of Connecticut\\ Storrs, CT 06269-1009}
\email{matthew.badger@uconn.edu}
\address{Department of Mathematics\\ The University of Tennessee\\ Knoxville, TN 37966}
\email{vvellis@utk.edu}

\begin{abstract} We investigate the H\"older geometry of curves generated by iterated function systems (IFS) in a complete metric space. A theorem of Hata from 1985 asserts that every connected attractor of an IFS is locally connected and path-connected. We give a quantitative strengthening of Hata's theorem. First we prove that every connected attractor of an IFS is $(1/s)$-H\"older path-connected, where $s$ is the similarity dimension of the IFS. Then we show that every connected attractor of an IFS is parameterized by a $(1/\alpha)$-H\"older curve for all $\alpha>s$. At the endpoint, $\alpha=s$, a theorem of Remes from 1998 already established that connected self-similar sets in Euclidean space that satisfy the open set condition are parameterized by $(1/s)$-H\"older curves. In a secondary result, we show how to promote Remes' theorem to self-similar sets in complete metric spaces, but in this setting require the attractor to have positive $s$-dimensional Hausdorff measure in lieu of the open set condition. To close the paper, we determine sharp H\"older exponents of parameterizations in the class of connected self-affine Bedford-McMullen carpets and build parameterizations of self-affine sponges. An interesting phenomenon emerges in the self-affine setting. While the optimal parameter $s$ for a self-similar curve in $\R^n$ is always at most the ambient dimension $n$, the optimal parameter $s$ for a self-affine curve in $\R^n$ may be strictly greater than $n$.
\end{abstract}

\maketitle

\vspace{-.45in}

\tableofcontents

\section{Introduction}

A special feature of one-dimensional metric geometry is the compatibility of intrinsic and extrinsic measurements of the length of a curve. Indeed, a theorem of Wa\.{z}ewski \cite{Wazewski} from the 1920s asserts that in a metric space a connected, compact set $\Gamma$ admits a continuous parameterization of finite total variation (intrinsic length) if and only if the set has finite one-dimensional Hausdorff measure $\Haus^1$ (extrinsic length). In fact, any curve of finite length admits  parameterizations $f:[0,1]\rightarrow\Gamma$, which are closed, Lipschitz, surjective, degree zero, constant speed, essentially two-to-one, and have total variation equal to $2\Haus^1(\Gamma)$; see Alberti and Ottolini \cite[Theorem 4.4]{AO-curves}. Unfortunately, this property---compatibility of intrinsic and extrinsic measurements of size---breaks down for higher-dimensional curves. While every curve parameterized by a continuous map of finite $s$-variation has finite $s$-dimensional Hausdorff measure $\Haus^s$, for each real-valued dimension $s>1$ there exist curves with $0<\Haus^s(\Gamma)<\infty$  that cannot be parameterized by a  continuous map of finite $s$-variation; e.g.~ see the ``Cantor ladders" in \cite[\S9.2]{BNV}. Beyond a small zoo of examples, there does not yet exist a comprehensive theory of curves of dimension greater than one. Partial investigations on H\"older geometry of curves from a geometric measure theory perspective include \cite{MM1993}, \cite{MM2000}, \cite{RZ16}, \cite{BV}, \cite{BNV}, and \cite{Balogh-Zust} (also see \cite{ident}). For example, in \cite{BNV} with Naples, we established a Wa\.{z}ewski-type theorem for higher-dimensional curves under an additional geometric assumption (flatness), which is satisfied e.g.~by von Koch snowflakes with small angles. The fundamental challenge is to develop robust methods to build good parameterizations.

Two well-known examples of higher-dimensional curves with H\"older parameterizations are the von Koch snowflake and the square (a space-filling curve). A common feature is that both examples can be viewed as the attractors of iterated function systems (IFS) in Euclidean space that satisfy the open set condition (OSC); for a quick review of the theory of IFS, see \S\ref{sec:prelim}. Remes \cite{Remes} proved that this observation is generic in so far as every connected self-similar set in Euclidean space of Hausdorff dimension $s\geq 1$ satisfying the OSC is a \emph{$(1/s)$-H\"older curve}, i.e.~ the image of a continuous map $f:[0,1]\rightarrow\R^n$ satisfying \begin{equation*} |f(x)-f(y)| \leq H|x-y|^{1/s}\quad\text{for all }x,y\in[0,1]\end{equation*} for some constant $H<\infty$. As an immediate consequence, for every integer $n\geq 2$ and real number $s\in(1,n]$, we can easily generate a plethora of examples of $(1/s)$-H\"older curves in $\R^n$ with $0<\Haus^s(\Gamma)<\infty$ (see Figure \ref{self-similar-scheme}).
\begin{figure}[th]
\begin{center}\includegraphics[width=1.5in]{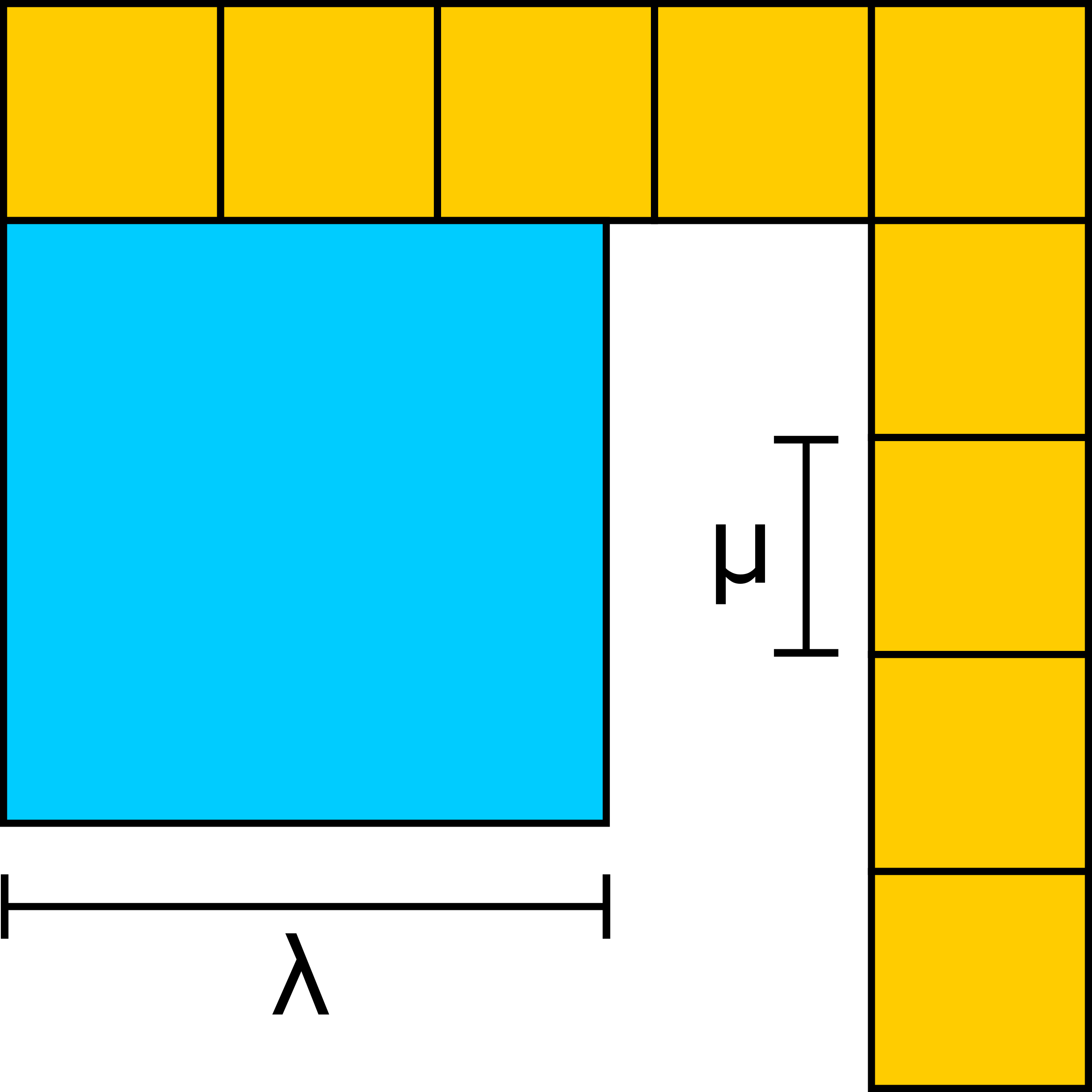}\quad\quad \includegraphics[width=1.5in]{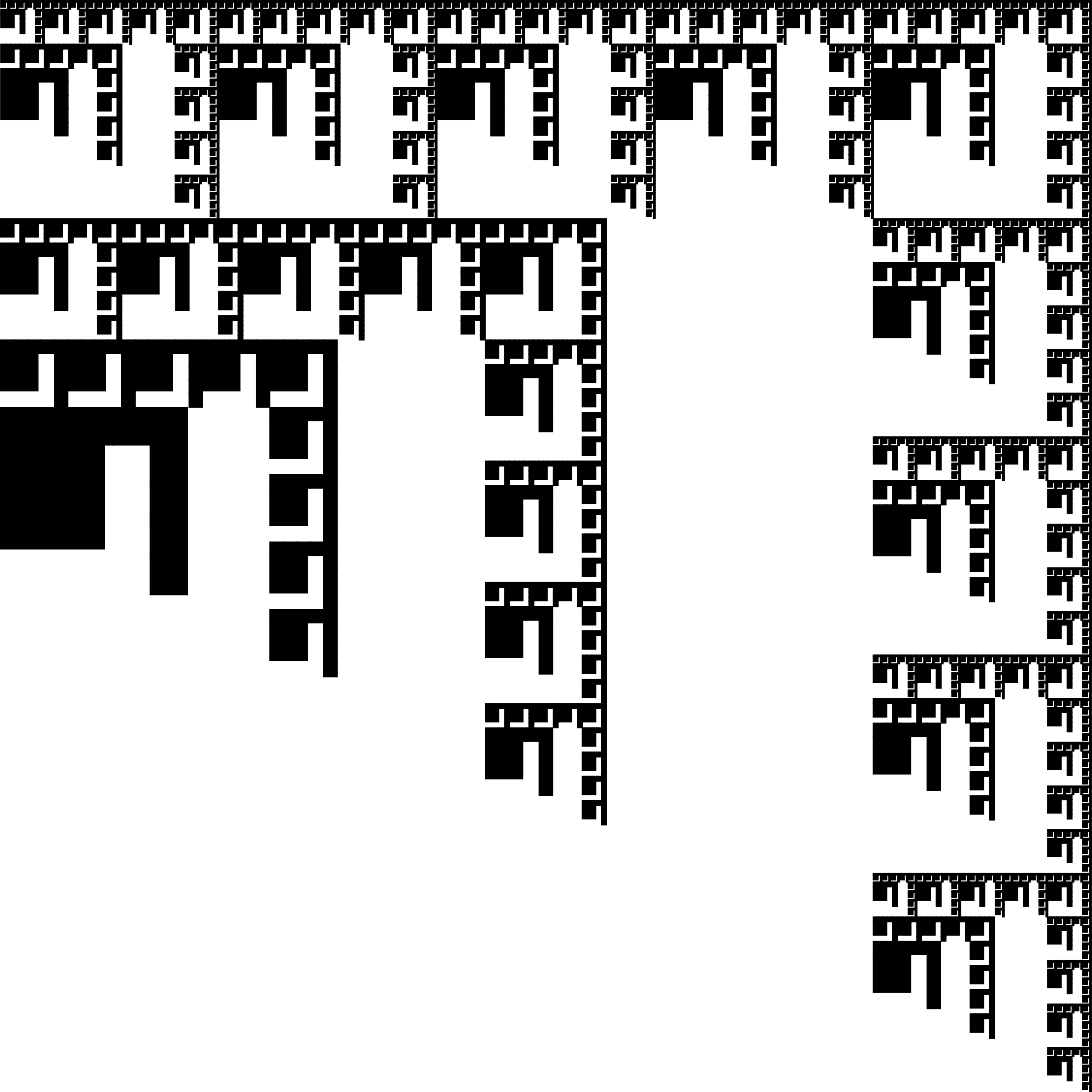}\end{center}
\caption{First and fourth iterations generating a self-similar $(1/s)$-H\"older curve $\Gamma$ in $\R^2$ with $0<\Haus^s(\Gamma)<\infty$; adjusting $\lambda\in[0,1-\mu]$ and $\mu=1/k$ (where $k\geq 2$ is an integer) yields examples of every dimension $s\in(1,2]$.}\label{self-similar-scheme}
\end{figure} However, with the view of needing a better theory of curves of dimension greater than one, we may ask whether Remes' method is flexible enough to generate H\"older curves under less stringent requirements, e.g.~can we parameterize self-similar sets in metric spaces or arbitrary connected IFS? The naive answer to this question is no, in part because measure-theoretic properties of IFS attractors in general metric or Banach spaces are less regular than in Euclidean space (see Schief \cite{Schief2}). Nevertheless, combining ideas from Remes \cite{Remes} and Badger-Vellis \cite{BV} (or Badger-Schul \cite{BS2}), we establish the following pair of results in the general metric setting. We emphasize that Theorems \ref{thm:connect} and \ref{thm:main} do not require the IFS to be generated by similarities nor do they require the OSC. In the statement of the theorems, extending usual terminology for self-similar sets, we say that the \emph{similarity dimension} of an IFS generated by contractions $\mathcal{F}$ is the unique number $s$ such that \begin{equation}\label{eq:sim} \sum_{\phi\in\mathcal{F}} (\Lip \phi)^s=1,\end{equation} where $\Lip \phi=\sup_{x\neq y} \dist(\phi(x),\phi(y))/\dist(x,y)$ is the Lipschitz constant of $\phi$.

\begin{thm}[H\"older connectedness]\label{thm:connect} Let $\mathcal{F}$ be an IFS over a complete metric space; let $s$ be the similarity dimension of $\mathcal{F}$. If the attractor $K_\mathcal{F}$ is connected, then every pair of points is connected in $K_\mathcal{F}$ by a $(1/s)$-H\"older curve.
\end{thm}

\begin{thm}[H\"older parameterization]\label{thm:main} Let $\mathcal{F}$ be an IFS over a complete metric space; let $s$ be the similarity dimension of $\mathcal{F}$. If the attractor $K_{\mathcal{F}}$ is connected, then $K_\mathcal{F}$ is a $(1/\a)$-H\"older curve for every $\a>s$.
\end{thm}

Early in the development of fractals, Hata \cite{Hata} proved that if the attractor $K_\mathcal{F}$ of an IFS over a complete metric space $X$ is connected, then $K_\mathcal{F}$ is locally connected and path-connected. By the Hahn-Mazurkiewicz theorem, it follows that if $K_\mathcal{F}$ is connected, then $K_\mathcal{F}$ is a \emph{curve}, i.e.~ $K_\mathcal{F}$ the image of a continuous map from $[0,1]$ into $X$. Theorems \ref{thm:connect} and \ref{thm:main}, which are our main results, can be viewed as a quantitative strengthening of Hata's theorem. We prove the two theorems directly, in \S\ref{sec:Holdcon}, without passing through Hata's theorem. A bi-H\"older variant of Theorem \ref{thm:connect} appears in Iseli and Wildrick's study \cite{Iseli-Wildrick} of self-similar arcs with quasiconformal parameterizations.

Roughly speaking, to prove Theorem \ref{thm:connect}, we embed the attractor $K_\mathcal{F}$ into $\ell_\infty$ and then construct a $(1/s)$-H\"older path between a given pair of points as the limit of a sequence of piecewise linear paths, mimicking the usual parameterization of the von Koch snowflake. Although the intermediate curves live in $\ell_\infty$ and not necessarily in $K_\mathcal{F}$, each successive approximation becomes closer to $K_\mathcal{F}$  in the Hausdorff metric so that the final curve is entirely contained in the attractor. Building the sequence of intermediate piecewise linear paths is a straightforward application of connectedness of an abstract word space associated to the IFS. The essential point to ensure the limit map is H\"older is to estimate the growth of the Lipschitz constants of the intermediate maps (see \S\ref{sec:Hparamcriterion} for an overview). Condition \eqref{eq:sim} gives us a natural way to control the growth of the Lipschitz constants, and thus, the similarity dimension determines the H\"older exponent of the limiting map (see \S\ref{sec:Holdcon}).  A similar technique allows us to parameterize the whole attractor of an IFS \emph{without branching} by a $(1/s)$-H\"older arc (see \S\ref{IFSsnowflakes}).

To prove Theorem \ref{thm:main}, we view the attractor $K_\mathcal{F}$ as the limit of a sequence of metric trees $\mathcal{T}_1\subset\mathcal{T}_2\subset\cdots$ whose edges are $(1/s)$-H\"older curves. Using condition \eqref{eq:sim}, one can easily show that \begin{equation}\label{e:treesup} S_\alpha:=\sup_{n} \sum_{E\in\mathcal{T}_n} (\diam E)^\alpha<\infty\quad\text{for all $\alpha>s$.}\end{equation} We then prove (generalizing a construction from \cite[\S2]{BV}) that \eqref{e:treesup} ensures $K_\mathcal{F}$ is a $(1/\alpha)$-H\"older curve for all $\alpha>s$. Unfortunately, because the constants $S_\alpha$ in \eqref{e:treesup} diverge as $\alpha\downarrow s$, we cannot use this method to obtain a H\"older parameterization at the endpoint. We leave the question of whether or not one can always take $\alpha=s$ in Theorem \ref{thm:main} as an open problem. The central issue is find a good way to control the growth of Lipschitz or H\"older constants of intermediate approximations for connected IFS \emph{with branching}.

For self-similar sets with positive $\Haus^s$ measure, we can build H\"older parameterizations at the endpoint in Theorem \ref{thm:main}. The following theorem should be attributed to Remes \cite{Remes}, who established the result for self-similar sets in Euclidean space, where the condition $\Haus^s(K_\mathcal{F})>0$ is equivalent to the OSC (see Schief \cite{Schief}). In metric spaces, it is known that $\Haus^s(K_\mathcal{F})>0$ implies the (strong) open set condition, but not conversely (see Schief \cite{Schief2}). A key point is that self-similar sets $K_\mathcal{F}$ with positive $\Haus^s$ measure are necessarily \emph{Ahlfors $s$-regular}, i.e.~$r^s\lesssim \Haus^s(K_\mathcal{F}\cap B(x,r))\lesssim r^s$ for all balls $B(x,r)$ centered on $K_\mathcal{F}$ with radius $0<r\lesssim \diam K_\mathcal{F}$. This fact is central to Remes' method for parameterizing self-similar sets with branching. See \S\ref{sec:proof} for a details.

\begin{thm}[H\"older parameterization for self-similar sets]\label{thm:remes} Let $\mathcal{F}$ be an IFS over a complete metric space that is generated by similarities; let $s$ be the similarity dimension of $\mathcal{F}$. If the attractor $K_\mathcal{F}$ is connected and $\Haus^s(K_\mathcal{F})>0$, then $K_\mathcal{F}$ is a $(1/s)$-H\"older curve.\end{thm}

As a case study, in \S\ref{sec:carpets}, to further illustrate the results above, we determine the sharp H\"older exponents in parameterizations of connected self-affine Bedford-McMullen carpets. We also build parameterizations of connected self-affine sponges in $\R^n$ (see Corollary \ref{cor:sponge}).

\begin{figure}[th]
\begin{center}\includegraphics[width=5in]{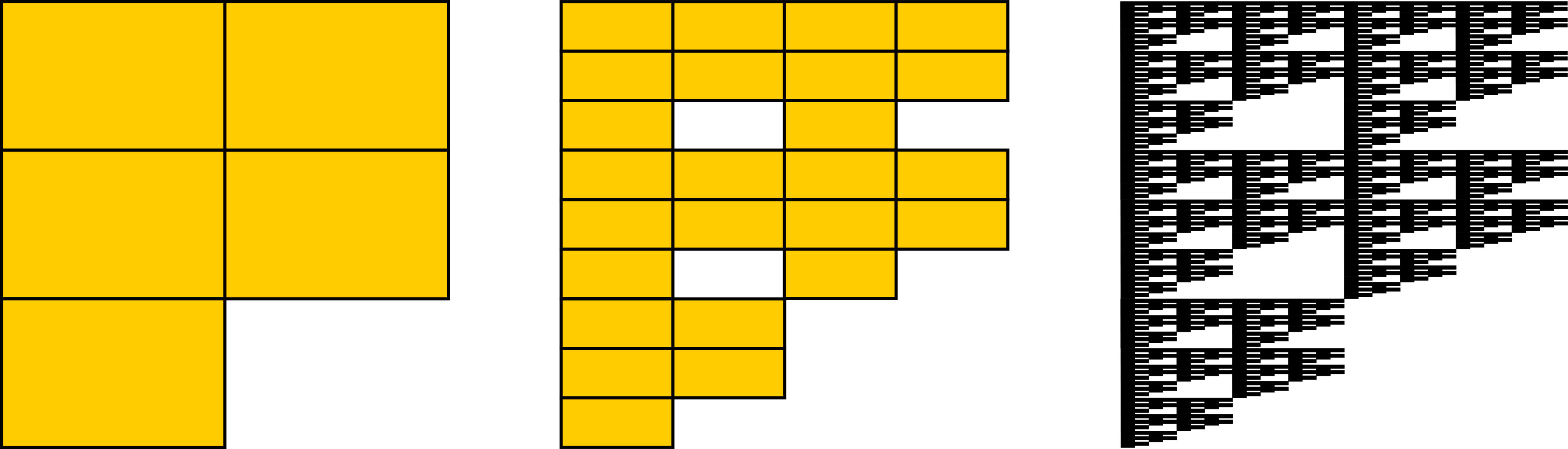}\end{center}
\caption{First, second, and fifth iterations of a Bedford-McMullen carpet $\Sigma$ that is a self-affine $(1/s)$-H\"older curve (with $\Haus^s(\Sigma)=0$) precisely when $s\geq \log_2(5)>2$.}\label{fig:mc23}
\end{figure}

\begin{thm}\label{thm:carpets2}Let $\Sigma\subset[0,1]^2$ be a connected Bedford-McMullen carpet (see \S \ref{sec:carpets}). \begin{itemize}
\item If $\Sigma$ is a point, then $\Sigma$ is (trivially) an $\alpha$-H\"older curve for all $\alpha>0$.
\item If $\Sigma$ is a line, then $\Sigma$ is (trivially) a 1-H\"older curve.
\item If $\Sigma$ is the square, then $\Sigma$ is (well-known to be) a $(1/2)$-H\"older curve.
\item Otherwise, $\Sigma$ is a $(1/s)$-H\"older curve, where $s$ is the similarity dimension of $\Sigma$.
\end{itemize} The H\"older exponents above are sharp, i.e.~ they cannot be increased.
\end{thm}

Of some note, the best H\"older exponent $1/s$ in parameterizations of a self-affine carpet can be strictly less than 1/2 (see Figure \ref{fig:mc23}). A similar phenomenon occurs for self-affine arcs in $\R^2$ (see \S\ref{sec:sharpsnow}). We interpret this as follows:

If the supremum of all exponents appearing in the set of H\"older parameterizations of a curve $\Gamma$ in a metric space is $1/s$, then we may say that $\Gamma$ has \emph{parameterization dimension} $s$. (If $\Gamma$ admits no H\"older parameterizations, then we say that $\Gamma$ has infinite parameterization dimension.) Intuitively, the parameterization dimension is a rough gauge of how fast a denizen of a curve must walk to visit every point in the curve. Every non-degenerate rectifiable curve has parameterization dimension 1 and a square has parameterization dimension 2. More generally, every self-similar curve in $\R^2$ that satisfies the OSC has parameterization dimension equal to its Hausdorff dimension. Theorem \ref{thm:carpets2} implies that there exist self-affine curves in $\R^2$ of arbitrarily large parameterization dimension.

\section{Preliminaries}\label{sec:prelim}

\subsection{Iterated function systems}\label{sec:IFS} Let $X$ be a complete metric space. A \emph{contraction} in $X$ is a Lipschitz map $\phi:X\rightarrow X$ with Lipschitz constant $\Lip\phi<1$, where \begin{equation}\Lip\phi:=\sup_{x\neq y} \frac{\dist(\phi(x),\phi(y))}{\dist(x,y)}\in[0,\infty].\end{equation} An \emph{iterated function system (IFS)} $\mathcal{F}$ is a finite collection of contractions in $X$. We say that $\mathcal{F}$ is \emph{trivial} if $\Lip\phi=0$ for every $\phi\in\mathcal{F}$; otherwise, we say that $\mathcal{F}$ is \emph{non-trivial}. The \emph{similarity dimension} $\sdim(\mathcal{F})$ of $\mathcal{F}$ is the unique number $s$ such that \begin{equation} \sum_{\phi\in\mathcal{F}}(\Lip \phi)^s=1,\end{equation} with the convention $\sdim(\mathcal{F})=0$ whenever $\mathcal{F}$ is trivial. Iterated function systems were introduced by Hutchinson \cite{Hutchinson} and encode familiar examples of fractal sets such as the Cantor ternary set, Sierpi\'nski carpet, and Sierpi\'nski gasket. For an extended introduction to IFS, see Kigami's \emph{Analysis on Fractals} \cite{Kigami}. Hutchinson's original paper as well as Hata's paper \cite{Hata} are gems in geometric analysis and excellent introductions to the subject in their own right.

\begin{thm}[Hutchinson \cite{Hutchinson}] If $\mathcal{F}$ is an IFS over a complete metric space, then there exists a unique compact set $K_\mathcal{F}$ in $X$ (the \emph{attractor} of $\mathcal{F}$) such that \begin{equation} K_\mathcal{F} = \bigcup_{\phi\in\mathcal{F}}\phi(K_\mathcal{F}).\end{equation} Furthermore, if $s=\sdim(\mathcal{F}$), then $\Haus^s(K_\mathcal{F})\leq (\diam K_\mathcal{F})^s<\infty$ and $\dim_H(K_\mathcal{F})\leq s$.\end{thm}

Above and below, the \emph{$s$-dimensional Hausdorff measure} $\Haus^s$ on a metric space is the Borel regular outer measure defined by \begin{equation}\Haus^s(E)=\lim_{\delta\downarrow 0} \inf\left\{\sum_{i=1}^\infty (\diam E_i)^s: E\subset \bigcup_{i=1}^\infty E_i,\, \sup_i \diam E_i\leq\delta\right\}\quad\text{for all }E\subset X.\end{equation} The \emph{Hausdorff dimension} $\dim_H(E)$ of a set $E$ in $X$ is the unique number given by \begin{equation}\dim_H(E):=\inf\{\alpha\in[0,\infty):\Haus^\alpha(E)<\infty\}=\sup\{\beta\in[0,\infty):\Haus^\beta(E)>0\}.\end{equation} For background on the fine properties of Hausdorff measures, Hausdorff dimension, and related elements of geometric measure theory, see Mattila's \emph{Geometry of Sets and Measures in Euclidean Spaces} \cite{Mattila}.

We say that an IFS $\mathcal{F}$ over a metric space $X$ satisfies the \emph{open set condition (OSC)} if there exists an open set $U \subset X$ such that
\begin{equation}\label{eq:OSC}
\phi(U) \subset U\quad\text{and}\quad
\phi(U)\cap \psi(U) = \emptyset \quad \text{for every $\phi,\psi \in \mathcal{F}$ with $\phi \neq \psi$}.
\end{equation} If there exists an open set $U\subset X$ satisfying (\ref{eq:OSC}), and in addition, $K_{\mathcal{F}} \cap U \neq \emptyset$, then we say that $\mathcal{F}$ satisfies the \emph{strong open set condition (SOSC)}. We say that the attractor $K_\mathcal{F}$ of an IFS $\mathcal{F}$ over $X$ is \emph{self-similar} if each $\phi\in\mathcal{F}$ is a \emph{similarity}, i.e.~ there exists a constant $0\leq L_\phi<1$ such that \begin{equation}\dist(\phi(x),\phi(y))=L_\phi\dist(x,y)\quad\text{for all }x,y\in X.\end{equation}

\begin{thm}[{Schief \cite{Schief}, \cite{Schief2}}]\label{thm:Schief} Let $K_\mathcal{F}$ be a self-similar set in $X$; let $s=\sdim(\mathcal{F})$. If $X$ is a complete metric space, then \begin{equation}\Haus^s(K_\mathcal{F})>0\Rightarrow \mathrm{SOSC}\Rightarrow \dim_H(K_\mathcal{F})=s. \end{equation} If $X=\R^n$, then \begin{equation} \Haus^s(K_\mathcal{F})>0\Leftrightarrow \mathrm{SOSC}\Leftrightarrow \mathrm{OSC}\Rightarrow \dim_H(K_\mathcal{F})=s\end{equation} Moreover, the implications above are the best possible (unlisted arrows are false).
\end{thm}

Given a metric space $X$, a set $E\subset X$, and radius $\rho>0$, let $N(E,\rho)$ denote the maximal number of disjoint closed balls with center in $E$ and radius $\rho$. Following Larman \cite{Larman}, $X$ is called a \emph{$\beta$-space} if for all $0<\beta<1$ there exist constants $1\leq N_\beta<\infty$ and $r_\beta>0$ such that $N(B,\beta r)\leq N_\beta$ for every open ball $B$ of radius $0<r\leq r_\beta$.

\begin{thm}[{Stella \cite{Stella}}]\label{thm:Stella} Let $K_\mathcal{F}$ be a self-similar set in $X$; let $s=\sdim(\mathcal{F})$. If $X$ is a complete $\beta$-space, then \begin{equation} \mathrm{SOSC}\Rightarrow \Haus^s(K_\mathcal{F})>0.\end{equation}\end{thm}

The following pair of lemmas are easy exercises, whose proofs we leave for the reader.

\begin{lem}\label{lem:sdimvsdim} Let $K_\mathcal{F}$ be a self-similar set in $X$; let $s=\sdim(\mathcal{F})$. If $\Haus^s(K_\mathcal{F})>0$, then $K_\mathcal{F}$ is Ahlfors $s$-regular, i.e.~there exists a constant $1\leq C<\infty$ such that \begin{equation}
C^{-1}r^{s} \leq \mathcal{H}^s(K_{\mathcal{F}}\cap B(x,r)) \leq C r^s\quad\text{for all }x\in K_\mathcal{F}\text{ and }0<r\leq \diam K_\mathcal{F}.\end{equation}
\end{lem}

\begin{lem}\label{lem:connIFS}
Let $\mathcal{F}$ be an IFS over a complete metric space. If $K_{\mathcal{F}}$ is connected, $\diam K_\mathcal{F}>0$, and $\phi \in \mathcal{F}$ has $\Lip(\phi) = 0$, then $K_{\mathcal{F}}$ agrees with the attractor of  $\mathcal{F}\setminus\{\phi\}$.
\end{lem}

\subsection{H\"older parameterizations}\label{sec:Hparamcriterion}

Let $s\geq 1$, let $X$ be a metric space, and let $f:[0,1]\to X$. We define the \emph{$s$-variation} of $f$ (\emph{over} $[0,1]$) by
\begin{equation}\label{eq:mass}
\|f\|_{s\text{-var}} := \left(\sup_{\mathcal{P}} \sum_{I \in \mathcal{P}} (\diam{f(I)})^s \right)^{1/s} \in [0,+\infty],
\end{equation}
where the supremum ranges over all finite interval partitions $\mathcal{P}$ of $[0,1]$. Here and below a \emph{finite interval partition} of an interval $I$ is a collection of (possibly degenerate) intervals $\{J_1,\dots,J_k\}$ that are mutually disjoint with $I=\bigcup_{i=1}^kJ_i$. We say that the map $f$ is \emph{$(1/s)$-H\"older continuous} provided that the associated \emph{$(1/s)$-H\"older constant} \begin{equation}\Hold_{1/s}(f):= \sup_{x\neq y} \frac{\dist(f(x),f(y))}{|x-y|^{1/s}}<\infty.\end{equation} By now, the following connection between continuous maps of finite $s$-variation and $(1/s)$-H\"older continuous maps is a classic exercise; for a proof and some historical remarks, see Friz and Victoir's \emph{Multidimensional Stochastic Processes as Rough Paths: Theory and Applications} \cite[Chapter 5]{FV-rough}. Although, we do not invoke Lemma \ref{prop:mass} directly below, behind the scenes many estimates that we carry out are motivated by trying to bound a discrete $s$-variation adapted to finite trees that we used in \cite[\S4]{BNV}.

\begin{lem}[{\cite[Proposition 5.15]{FV-rough}}] \label{prop:mass}
Let $s\geq1$ and let $f:[0,1]\to X$ be continuous.
\begin{enumerate}
\item If $f$ is $(1/s)$-H\"older, then $\|f\|_{s\text{-var}} \leq \Hold_{1/s}{f}$.
\item If $\|f\|_{s\text{-var}} < \infty$, then there exists a continuous surjection $\psi:[0,1] \to [0,1]$ and a $(1/s)$-H\"older map $F:[0,1]\to X$ such that $f = F\circ\psi$ and $\Hold_{1/s}{F} \leq \|f\|_{s\text{-var}}$.
\end{enumerate}
\end{lem}

The standard method to build a H\"older parameterization of a curve in a Banach space that we employ below is to exhibit the curve as the pointwise limit of a sequence of Lipschitz curves with controlled growth of Lipschitz constants. We will use this principle frequently, and also on one occasion in \S\ref{sec:Holdcon}, the following extension where the intermediate maps are H\"older continuous.

\begin{lem}\label{l:LipHold} Let $1\leq t <s$, $M>0$, $0<\xi_1\leq \xi_2<1$, $\alpha>0$, $\beta>0$, and $j_0\in\Z$. Let $(X,|\cdot|)$ be a Banach space. Suppose that $\rho_j$ $(j\geq j_0)$ is a sequence of scales and  $f_j:[0,M]\rightarrow X$ $(j\geq j_0)$ is a sequence of $(1/t)$-H\"older maps satisfying \begin{enumerate}
\item $\rho_{j_0}=1$ and $\xi_1 \rho_j \leq \rho_{j+1}\leq \xi_2 \rho_j$ for all $j\geq j_0$,
\item $|f_j(x)-f_j(y)| \leq A_j|x-y|^{1/t}$ for all $j\geq j_0$, where $A_j \leq \alpha \rho_j^{1-s/t}$, and
\item $|f_j(x)-f_{j+1}(x)| \leq B_j$ for all $j\geq j_0$, where $B_j \leq \beta \rho_j$.\end{enumerate} Then
$f_j$ converges uniformly to a map $f:[0,M]\rightarrow X$ such that $$|f(x)-f(y)|\leq H |x-y|^{1/s}\quad\text{for all $x,y\in[0,M]$},$$ where $H$ is a finite constant depending on at most $\max(M,M^{-1})$, $\xi_1$, $\xi_2$, $\alpha$, and $\beta$. In particular, we may take \begin{equation}\label{e:H} H=\frac{\alpha}{\xi_1}\max(1,M)+\frac{2\beta}{\xi_1(1-\xi_2)}\max(1,M^{-1}).\end{equation} \end{lem}

\begin{proof} The statement and proof in the case $t=1$ is written in full detail in \cite[Lemma B.1]{BNV}. The proof of the general case follows \emph{mutatis mutandis}.\end{proof}

\begin{cor}\label{c:LipHold} Let $f_{j_0},\dots,f_{j_1}$ be a finite sequence of functions and $\rho_{j_0},\dots,\rho_{j_1}$ be a finite sequence of scales satisfying the hypothesis of Lemma \ref{l:LipHold}, i.e.~assume that (1) and (3) hold for all $j_0\leq j\leq j_1-1$ and (2) holds for all $j_0\leq j\leq j_1$. Then the function $f_{j_1}$ is $(1/s)$-H\"older continuous with $\Hold_{1/s} f_{j_1}\leq H$, where $H$ is given by \eqref{e:H}.\end{cor}

\begin{proof} Extend the sequence of functions $f_{j_0},\dots,f_{j_1}$ to an infinite sequence by setting $f_j\equiv f_{j_1}$ for all $j>j_1$. Also choose any extension of the sequence of scales $\rho_{j_0},\dots,\rho_{j_1}$ satisfying (1). Then the full sequence $(f_j,\rho_j)_{j_0}^\infty$ satisfies the hypothesis of the lemma with $A_j\equiv A_{j_1}$ and $B_j\equiv 0$ for all $j>j_1$. Therefore, $f_{j_1}\equiv \lim_{j\rightarrow\infty} f_j$ is $(1/s)$-H\"older with $\Hold_{1/s} f_{j_1}\leq H$. \end{proof}

\subsection{Words}\label{sec:words}
Suppose we are given an IFS $\mathcal{F}=\{\phi_1,\dots,\phi_k\}$ over a complete metric space $X$ such that $\Lip \phi_i>0$ for all $1\leq i\leq k$. Set $s:=\sdim(\mathcal{F})$, and for each $i\in\{1,\dots,k\}$, set $L_i := \Lip(\phi_i)$. Relabeling, we may assume without loss of generality that \begin{equation}\label{L-increasing} 0< L_1 \leq \cdots \leq L_k < 1.\end{equation} By definition of the similarity dimension, we have $L_1^s+\dots+L_k^s=1$.

Define the \emph{alphabet} $A=\{1,\dots,k\}$. Let $A^n=\{i_1\cdots i_n:i_1,\dots,i_n\in A\}$ denote the set of \emph{words} in $A$ and of \emph{length} $n$. Also let $A^0 = \{\epsilon\}$ denote the set containing the \emph{empty word} $\epsilon$ of length 0. Let $A^*= \bigcup_{n\geq 0}A^n$ denote the set of \emph{finite words} in $A$. Given any finite word $w\in A^*$ and length $n\in\N$, we assign \begin{equation} A^*_w:= \{u \in A^*: u=wv\}\quad\text{and}\quad A^n_w = \{wv\in A^*_w : |wv|=n\}.\end{equation} The set $A^*_w$ can be viewed in a natural way as a tree with root at $w$. We also let $A^{\N}$ denote the set of \emph{infinite words} in $A$. Given an infinite word $w=i_1 i_2\cdots \in A^{\N}$ and integer $n\geq 0$, we define the \emph{truncated word} $w(n)=i_1\cdots i_n$ with the convention that $w(0) = \epsilon$.

We now organize the set of finite words in $A$, according to the Lipschitz norms of the associated contractions! This will be used pervasively throughout the rest of the paper. For each word $w=i_1\cdots i_n \in A^*$, define the map \begin{equation}\phi_w:=\phi_{i_1}\circ\dots\circ\phi_{i_n}\end{equation} and the weight \begin{equation}L_w := L_{i_1}\cdots L_{i_n}.\end{equation} By convention, for the empty word, we assign $\phi_\epsilon:=\mathrm{Id}_X$ and $L_\epsilon:=1$. For all $w\in A^*$, define the \emph{cylinder} $K_w$ to be the image of the attractor $K:=K_\mathcal{F}$ under $\phi_w$, \begin{equation}K_w:=\phi_w(K).\end{equation} Note that $L_{uv}=L_uL_v$ for every pair of words $u$ and $v$, where $uv$ denotes the concatenation of $u$ followed by $v$. For each $\d\in (0,1)$, define \begin{equation}
A^*(\d) := \{i_1\cdots i_n \in A^* : \text{$n\geq 1$ and $L_{i_1}\cdots L_{i_n} < \d \leq L_{i_1}\cdots L_{i_{n-1}}$} \}\end{equation}
with the convention $L_1\cdots L_{i_{n-1}} = 1$ if $n=1$. Also define $A^*(1):=\{\epsilon\}$. Finally, given any finite word $w\in A^*$, set $A^*_w(\delta):=A^*_w\cap A^*(\delta)$.

\begin{lem}\label{lem:partition}
Given finite words $w\in A^*$ and $w' = wi_1\cdots i_n$ and a number $L_{w'}<\d \leq L_w$, there exists a unique finite word $u=wi_1\cdots i_m$ ($m\leq n$) such that $u\in A^*_w(\d)$.
\end{lem}

\begin{proof}
Existence of $u$ follows from the fact that the sequence $a_n = L_{wi_1\cdots i_n}$ is decreasing. Uniqueness of $u$ follows from the fact that if $wi_1\cdots i_m \in A^*_w(\d)$, then for every $l<m$, $L_{wi_1\cdots i_l} \geq \d$, whence $wi_1\cdots i_l \not\in A^*_w(\d)$.
\end{proof}

\begin{lem}\label{l:sum-Ls} For every finite word $w\in A^*$ and number $0<\delta \leq L_w$, \begin{equation} \label{sim-sum} \sum_{w'\in A^*_w(\delta)} L_{w'}^s=L_w^s.\end{equation} \end{lem}

\begin{proof} By \eqref{L-increasing}, we can choose $N\in\N$ sufficient large so that $L_{u}<L_1\delta$ for all words $u\in A^N$ (any integer $N>\log_{L_k}(L_1\delta)$ will suffice). In particular, if $wv=w{v_1\dots v_n}\in A_w^N$, then $wv\not\in A_w^*(\delta)$ (since $L_{wv_1\dots v_{n-1}}<\delta$) but $wv$ has an ancestor $wv_1\dots v_m\in A_w^*(\delta)$ by Lemma \ref{lem:partition}. Hence the subtree $T=\bigcup_{l=|w|}^N A_w^{l}$ of $A_w^*$ contains $A^*_w(\delta)$. To establish \eqref{sim-sum}, we repeatedly use the defining condition $L_1^s+\dots+L_k^s=1$ for the similarity dimension, first working ``down" the tree $T$ from each word $w'\in A^*(\delta)$ to its descendants in $A^N_w$ and then working ``up" the tree $T$ level by level: \begin{equation*}\sum_{w'\in A^*_w(\delta)} L_{w'}^s = \sum_{w'' \in A^N_w} L_{w''}^s = \sum_{w''' \in A^{N-1}_w} L_{w'''}^s=\dots= L_w^s. \qedhere\end{equation*}
\end{proof}

\begin{lem}\label{6}
For all $0<R\leq 1$, $w\in A^*(R)$, and $0<r\leq L_w$,
\begin{equation}\label{e:cardest}L_1^{s} (R/r)^{s} < \card A^*_w(r) < L_1^{-s} (R/r)^{s}.\end{equation} In particular, if $0<r\leq 1$, then \begin{equation}\label{e:cardest2} L_1^s r^{-ms}<\card A^*(r^m)<L_1^{-s} r^{-ms}\quad\text{for all }m\in\N.\end{equation}
\end{lem}

\begin{proof}
Fix $0<R\leq 1$, $w\in A^*(R)$, and $0<r\leq L_w$. Then $L_w<R\leq L_w/L_1$, and similarly, for all $w'\in A_w^*(r)$, we have $L_{w'}<r\leq L_{w'}/L_{1}$. By Lemma \ref{l:sum-Ls},
\[ L_1^sr^s(\card A_w^*(r)) \leq \sum_{w'\in A^*_w(r)}L_{w'}^s= L_w^s < R^s. \]
Similarly,
\[ r^s(\card A_w^*(r)) > \sum_{w'}L_{w'}^s = L_w^s \geq L_1^s R^s \] This establishes \eqref{e:cardest}. To derive \eqref{e:cardest2}, simply take $0<r\leq 1=R$ and $w=\epsilon$.
\end{proof}

\section{H\"older connectedness of IFS attractors}\label{sec:Holdcon}

In this section, we first prove Theorem \ref{thm:connect}, and afterwards, we derive Theorem \ref{thm:main} as a corollary. To that end, for the rest of this section, fix an IFS $\mathcal{F}=\{\phi_1,\dots,\phi_k\}$ over a complete metric space $(X,d)$ whose attractor $K:=K_{\mathcal{F}}$ is connected and has positive diameter. Set $s:=\sdim(\mathcal{F})$, and for each $i\in\{1,\dots,k\}$, set $L_i := \Lip(\phi_i)$. By Lemma \ref{lem:connIFS}, we may assume without loss of generality that \begin{equation}\label{L-increasing2} 0< L_1 \leq \cdots \leq L_k < 1.\end{equation} In particular, we may adopt the notation, conventions, and lemmas in \S\ref{sec:words}.

\subsection{H\"older connectedness (Proof of Theorem \ref{thm:connect})}

\begin{lem}[chain lemma]\label{lem:conn} Assume that $K_\mathcal{F}$ is connected.
Let $w\in A^*$ and $0<\d< L_w$. If $x,y \in K_w$, then there exist distinct words $w_1,\dots, w_n \in A^*_w(\d)$ such that $x \in K_{w_1}$, $y\in K_{w_n}$, and $K_{w_i}\cap K_{w_{i+1}} \neq \emptyset$ for all $i\in\{1,\dots,n-1\}$.
\end{lem}

\begin{proof}
We first remark that $K_w = \bigcup_{u\in A^*_w(\d)} K_u$ by Lemma \ref{lem:partition}. Define
\[ E_1 := \{u\in A^{*}_w(\d) : x \in K_u\}.\]
Assuming we have defined $E_1,\dots,E_i \subset A^*_w(\d)$ for some $i\in\N$, define
\[ E_{i+1} := \{u \in A^*_w(\d)\setminus E_i : K_u \cap K_v \neq \emptyset \text{ for some $v \in E_i$}\}.\]
Because $K_w$ is connected (since $K_\mathcal{F}$ is connected), if $\bigcup_{i=1}^j E_i\neq A^*_w(\d)$, then $E_{j+1}\neq \emptyset$. Since $A^*_w(\delta)$ is finite, it follows that $\bigcup_{i=1}^N E_i = A^*_w(\d)$ for some $N \in \N$.

Choose a word $v\in A^*_w(\delta)$ such that $y\in K_{v}$. Then $v\in E_n$ for some $1\leq n\leq N$. Label $v=:w_n$. By design of the sets $E_i$, we can find a chain of distinct words $w_1,\dots, w_n$ with $K_{w_i}\cap K_{w_{i+1}}$ for all $1\leq i\leq n-1$. Finally, $x\in K_{w_1}$, because $w_1\in E_1$.
\end{proof}

Theorem \ref{thm:connect} is a special case of the following more precise result (take $w$ to be the empty word). Recall that a metric space $(X,d)$ is \emph{quasiconvex} if any pair of points $x$ and $y$ can be joined by a Lipschitz curve $f:[0,1]\rightarrow X$ with $\Lip(f)\lesssim_X d(x,y)$. By analogy, the following proposition may be interpreted as saying that  connected attractors of IFS are ``$(1/s)$-H\"older quasiconvex".

\begin{prop}\label{lem:Holderconn} For any $w\in A^*$ and $x,y \in K_w$, there exists a $(1/s)$-H\"older continuous map $f:[0,L_w^s] \to K_w$ with $f(0)=x$, $f(L_w^s)=y$, and $\Hold_{1/s}{f} \lesssim_{s,L_1} \diam K$.
\end{prop}

\begin{proof} By rescaling the metric on $X$, we may assume without loss of generality that $\diam K=1$. Furthermore, it suffices to prove the proposition for $w=\epsilon$ and $K_w=K$. For the general case, fix $w\in A^*$ and $x,y \in K_w$. Choose $x',y' \in K$ such that $\phi_w(x') = x$ and $\phi_w(y') = y$. Define $$\zeta_w: [0,L_w^s] \to [0,1],\qquad \zeta_w(t) = (L_w)^{-s}t\quad\text{for all }t\in[0,L_w^s].$$ If the proposition holds for $w=\epsilon$, then there exists a $(1/s)$-H\"older map $g : [0,1] \to K$ with $g(0)=x'$, $g(1)=y'$, and $\Hold_{1/s}{g} \lesssim_{s,L_1} 1$. Then the map $f \equiv \phi_w\circ g\circ\zeta_w : [0,L_w^s] \to K_w$ plainly  satisfies $f(0) =x$ and $f(L_w^s) = y$. Moreover, for any $p,q \in [0,L_w^s]$,
\[ d(f(p),f(q)) \leq L_w\, d(g(\zeta_w(p)),g(\zeta_w(q))) \lesssim_{s,L_1}L_w|\zeta_w(p)-\zeta_w(q)|^{1/s} = |p-q|^{1/s}.\] Thus, $\Hold_{1/s}(f)\lesssim_{s,L_1} 1$, independent of the word $w$.

To proceed, observe that by the Kuratowski embedding theorem, we may view $K$ as a subset of $\ell_\infty$, whose norm we denote by $|\cdot|_{\infty}$. Fix any $r>0$ with $L_1\lesssim r\leq L_1$ (which ensures that $r^{m+1}\leq L_1r^m\leq L_w$ whenever $w\in A^*(r^m)$) and fix $x,y \in K$. The map $f$ will be a limit of piecewise linear maps $f_n : [0,1] \to \ell_{\infty}$. In particular, for each $m\in \N$, we will construct a subset $\mathcal{W}_m \subset A^*(r^m)$, a family of nondegenerate closed intervals $\mathscr{E}_m$, and a continuous map $f_m:[0,1] \to \ell_{\infty}$ satisfying the following properties:
\begin{enumerate}
\item[{(P1)}] The intervals in $\mathscr{E}_m$ have mutually disjoint interiors and their union $\bigcup\mathscr{E}_m = [0,1]$. Furthermore, $f_m(0)=x$ and $f_m(1)=y$.
\item[{(P2)}] For each $I\in \mathscr{E}_m$, $f_m|I$ is linear and there exists $u \in \mathcal{W}_m$ such that $f_m(\partial I) \subset K_{u}$ and $|I| \geq L_{u}^s$. Moreover, if $I,I' \in \mathscr{E}_m$ are distinct, then the corresponding words $u,u' \in \mathcal{W}_m$ are also distinct.
\item[{(P3)}] For each $I \in \mathscr{E}_{m+1}$, there exists $J \in \mathscr{E}_m$ such that $f_{m+1}|\partial J = f_m|\partial J$. Moreover, $|f_m(p)-f_{m+1}(p)|_{\infty} < 3r^m$ for all $p\in [0,1]$.
\end{enumerate}
Let us first see how to complete the proof, assuming the existence of family of such maps. On one hand, property (P3) gives
\begin{equation}\label{eq:1}
\|f_m - f_{m+1}\|_{\infty} < 3r^m.
\end{equation}
On the other hand, by property (P2), $|I| \geq L_1^sr^{ms}$ and $\diam{f_m(I)} < r^m$ for all $I\in \mathscr{E}_m$. Therefore, for all $p,q\in [0,1]$,
\begin{equation}\label{eq:2}
|f_m(p) - f_m(q)|_{\infty} \leq L_1^{-s} r^{m(1-s)}|p-q|.
\end{equation}
By (\ref{eq:1}), (\ref{eq:2}), and Lemma \ref{l:LipHold}, the sequence $(f_m)_{m=1}^\infty$ converges uniformly to a $(1/s)$-H\"older map $f:[0,1] \to \ell_{\infty}$ with $f(0)=x$, $f(1) =y$, and $\Hold_{1/s}{f} \lesssim_{s,L_1,r} 1\simeq_{s,L_1} 1$. Finally, by (P2) and \eqref{eq:2}, $$\dist(f_m(p),K) \lesssim_{s,L_1} r^m\quad\text{for all $m\in\N$ and $p\in[0,1]$}.$$ Therefore, $f([0,1]) \subset K$ and the proposition follows.

It remains to construct $\mathcal{W}_m$, $\mathscr{E}_m$, and $f_m$ satisfying properties (P1), (P2), and (P3). The construction is in an inductive manner.

By Lemma \ref{lem:conn}, there is a set $\mathcal{W}_1 = \{w_1,\dots,w_n\}$ of distinct words in $A^*(r)$, enumerated so that $x\in K_{w_1}$, $y\in K_{w_n}$, and $K_{w_{i}}\cap K_{w_{i+1}} \neq \emptyset$ for $i\in\{1,\dots,n-1\}$. For each $i\in\{1,\dots,n-1\}$, choose $p_i \in K_{w_i}\cap K_{w_{i+1}}$. To proceed, define $\mathscr{E}_1=\{I_1,\dots, I_n\}$ to be closed intervals in $[0,1]$ with disjoint interiors, enumerated according to the orientation of $[0,1]$, whose union is $[0,1]$, and such that $|I_j| \geq L_{w_i}^s$ for all $i\in \{1,\dots,n\}$.
We are able to find such intervals, since by Lemma \ref{l:sum-Ls},
\[ 1 = \sum_{u\in A^*(r)}L_u^s \geq \sum_{u \in \mathcal{W}_1} L_{u}^s.\]
Next, define $f_1 : [0,1] \to \ell_{\infty}$ in a continuous fashion so that $f_1$ is linear on each $I_i$ and:
\begin{enumerate}
\item $f_1(0) = x$ and $f_1(I_1)$ is the segment that joins $x$ with $p_1$;
\item $f_1(1) = y$ and $f_1(I_n)$ is the segment that joins $p_{n-1}$ with $y$; and,
\item for $j\in\{2,\dots,n-1\}$, if any, $f_1(I_j)$ is the segment that joins $p_{j-1}$ with $p_{j}$.
\end{enumerate}

Suppose that for some $m\in\N$, we have defined $\mathcal{W}_m \subset A^*(r^m)$, a collection $\mathscr{E}_m$, and a piecewise linear map $f_m : [0,1] \to \ell_{\infty}$ that satisfy (P1)--(P3). For each $I \in \mathscr{E}_m$, we will define a collection of intervals $\mathscr{E}_{m+1}(I)$ and a collection of words $\mathcal{W}_{m+1}(I) \subset A^*(r^{m+1})$. We then set $\mathscr{E}_{m+1} = \bigcup_{I \in \mathscr{E}_m}\mathscr{E}_{m+1}(I)$ and $\mathcal{W}_{m+1} =  \bigcup_{I \in \mathscr{E}_m}\mathcal{W}_{m+1}(I)$.
In the process, we will also define $f_{m+1}$. To proceed, suppose that $I\in \mathscr{E}_m$, say $I = [a,b]$, with $I$ corresponding to the word $w\in \mathcal{W}_m$. Since $K$ is connected, by Lemma \ref{lem:conn}, there exist distinct words $\mathcal{W}_{m+1}(I) = \{w_1,\dots, w_l\} \subset A^*_w(r^{m+1})$ such that $f_m(a) \in K_{w_1}$, $f_m(b)\in K_{w_l}$, and $K_{w_j}\cap K_{w_{j+1}} \neq \emptyset$ for all  $j\in\{1,\dots, l-1\}$. Let $\mathcal{E}_{m+1}(I) = \{I_1,\dots, I_l\}$ be closed intervals in $I$ with mutually disjoint interiors, enumerated according to the orientation of $I$, whose union is $I$, and such that $a\in I_1$, $b\in I_l$ and $|I_j| \geq L_{w_j}^s$ for all $j\in\{1,\dots,l\}$. We are able to find such intervals, since by our inductive hypothesis and Lemma \ref{l:sum-Ls},
\begin{equation*}\label{eq:partition}
|I| \geq L_w^s = \sum_{u\in A^*_w(r^{m+1})}L_u^s \geq \sum_{i=1}^l L_{w_i}^s.
\end{equation*}
For each $j\in \{1,\dots,l-1\}$, choose $p_j \in K_{w_j}\cap K_{w_{j+1}}$.

With the choices above, now define $f_{m+1}|I : I \to \ell_{\infty}$ in a continuous fashion so that $f_{m+1}|J$ is linear for each $J\in\mathscr{E}_{m+1}(I)$ and:
\begin{enumerate}
\item $f_{m+1}(a) = f_m(a)$ and $f_{m+1}(I_1)$ is the segment that joins $y$ with $p_1$;
\item $f_{m+1}(b) = f_m(b)$ and $f_{m+1}(I_l)$ is the segment that joins $p_{l-1}$ with $f_m(b)$; and,
\item for $j\in\{2,\dots,l-1\}$ (if any), $f_{m+1}(I_j)$ is the segment that joins $p_{j-1}$ with $p_{j}$.
\end{enumerate} Properties (P1), (P2), and the first claim of (P3) are  immediate. To verify the second claim of (P3), fix $z\in [0,1]$. By (P1), there exists $I\in \mathscr{E}_{m+1}$ such that $z\in I$. Let $J$ be the unique element of $\mathscr{E}_m$ such that $I\subset J$.
Then there exists $w\in A^*(r^m)$ such that $I\in \mathscr{E}_{m+1}(J)$ and $f_{m}(\partial J) \subset K_w$. Since $f_{m+1}(\partial I) \subset K_{u}$ for some $u \in A^*_w(r^{m+1})$, we have that $f_{m+1}(\partial I)$. Let $y_1\in \partial I$ and $y_2\in \partial J$. We have
\begin{align*}
|f_m(z)&-f_{m+1}(z)|_{\infty} \\ &\leq |f_m(z) - f_m(y_2)|_{\infty} + |f_m(y_2)-f_{m+1}(y_1)|_{\infty} + |f_{m+1}(y_1) - f_{m+1}(z)|_{\infty}\\
&\leq 3\diam{K_w}<3 r^m. \qedhere
\end{align*}
\end{proof}

\subsection{H\"older parameterization (Proof of Theorem \ref{thm:main})}

The proof of Theorem \ref{thm:main} is modeled after the proof of \cite[Theorem 2.3]{BV}, which gave a criterion for the set of leaves of a ``tree of sets" in Euclidean space to be contained in a H\"older curve. Here we view the attractor $K_\mathcal{F}$ as the set of leaves of a tree, whose edges are H\"older curves.

\begin{proof}[Proof of Theorem \ref{thm:main}]

Rescaling the metric $d$, we may assume for the rest of the proof that $\diam{K} = 1$. Fix $q\in K$, and for each $w\in A^*$, set $q_w := \phi_w(q)$ with the convention $q_{\epsilon} = q$. Fix $\a>s=\sdim{\mathcal{F}}$ and fix $L_1\lesssim r \leq L_1$ (once again ensuring that $r^{m+1}\leq L_1r^m\leq L_w$ for all $w\in A^*(r^m)$).
By Lemma \ref{6}, for every integer $m\geq 0$, the set $A(r^m)$ has fewer than $L_1^{-s}r^{-ms}$ words, and moreover, for every $w\in A^*(r^m)$, the set $A_w^*(r^{m+1})$ has at least $1$ and fewer than $L_1^{-s}r^{-s}$ words. Since $r\simeq L_1$,
\begin{equation}\begin{split}\label{eq:3}
\sum_{m=0}^{\infty}&\sum_{w\in A^*(r^m)} \sum_{u \in A^*_w(r^{m+1})} d(q_w,q_u)^{\a} \leq  \sum_{m=0}^{\infty}\sum_{w\in A^*(r^m)} \sum_{u \in A^*_w(r^{m+1})} L_w^{\a} < L_1^{-s}r^{-s} \sum_{m=0}^{\infty}\sum_{w\in A^*(r^m)} L_w^{\a} \\ &\quad< L_1^{-s}r^{-s} \sum_{m=0}^{\infty}\sum_{w\in A^*(r^m)} r^{\a m}
\leq L_1^{-2s}r^{-s} \sum_{m=0}^{\infty} r^{(\a-s)m}\lesssim_{L_1,s,\a} 1.
\end{split}\end{equation}

Below we call the elements of $A_w^*(r^{m+1})$ the \emph{children} of $w\in A^*(r^m)$, and we call $w$ their \emph{parent}; if $u\in A^*_w(r^{m+1})$, then we write $w =: p(u)$.
For each $w\in A^*(r^m)$ and $u\in A^*_w(r^{m+1})$, let $f_{w,u} : [0,L_w^s] \to K_w$ be the $(1/s)$-H\"older map with $f_{w,u}(0) = q_w$ and $f_{w,u}(L_w^s) = q_u$ given by Proposition \ref{lem:Holderconn}. Let also $\gamma_{w,u}$ be the image of $f_{w,u}$. We can write $K$ as the closure of the set
\[\G_{\circ} := \bigcup_{m=0}^{\infty}\bigcup_{w\in A^*(r^m)}\bigcup_{u\in A_w^*(r^{m+1})}\g_{w,u}.\]
For each integer $m\geq 0$ and $w\in A^*(r^m)$ define
\[ M_{w} := 2\sum_{j=m+1}^{\infty} \sum_{u \in A^*_w(r^{j})}
L_{p(u)}^{\a} \lesssim_{L_1,s,\a} r^{m\a},\]
where we sum over all descendants of $w$. Setting $M:= M_{\epsilon}$, by (\ref{eq:3}), we have that $M\lesssim_{L_1,s,\a} 1$. We will construct a $(1/\a)$-H\"older continuous surjective map $F: [0,M] \to K$ by defining a sequence $F_m : [0,M] \to K$ ($m\in\N$) whose limit is $F$ and whose image is the truncated tree
\[ \G_m := \bigcup_{i=0}^{m-1}\bigcup_{w\in A^*(r^{i})}\bigcup_{u\in A_w^*(r^{i+1})}\g_{w,u}.\]

\begin{lem}\label{lem:tree-param}
For each $m\in\N$, there exist two collections $\mathscr{B}_m$, $\mathscr{N}_m$ of nondegenerate closed intervals in $[0,1]$, a bijection $\eta_m : \mathscr{N}_m \to A^*(r^m)$, and a map $F_m:[0,M] \to \G_m$ with the following properties.
\begin{enumerate}
\item[{(P1)}] The families $\mathscr{N}_m$ and $\mathscr{B}_m$ are disjoint, the elements in $\mathscr{N}_m\cup\mathscr{B}_m$ have mutually disjoint interiors, and $\bigcup(\mathscr{N}_m\cup \mathscr{B}_m) = [0,M]$. Moreover, $F_m([0,M]) = \G_m$.
\item[{(P2)}] If $I\in \mathscr{N}_{m+1}$, then there is $J\in\mathscr{N}_m$ such that $I\subset J$ and $\eta_{m+1}(I) \in A^*_{\eta_m(J)}(r^{m+1})$. Conversely, if $J\in\mathscr{N}_n$, then there exist $J_1\in \mathscr{N}_{m+1}$ and $J_2\in\mathscr{B}_{m+1}$ such that $J_1\subset I$ and $J_2\subset I$ and $\card\{ I \in \mathscr{B}_{m+1}\cup\mathscr{N}_{m+1} : I\subset J\} \leq L_1^{-s}r^s$.
\item[{(P3)}] If $I\in\mathscr{B}_{m+1}$, then either $I \in \mathscr{B}_{m}$ or there exists $J\in \mathscr{N}_m$ such that $I\subset J$. Conversely, $\mathscr{B}_m\subset\mathscr{B}_{m+1}$.
\item[{(P4)}] For each $I \in \mathscr{N}_m$, $|I| = M_{\eta_m(I)}$, $F_m|I$ is constant and equal to $q_{\eta(I)}$ and $F_{m+1}|\partial I = F_m|\partial I$.
\item[{(P5)}] For each $I\in \mathscr{B}_m$, there exists $w\in A^*(r^{m-1})$ and $u\in A^*_w(r^m)$ such that $|I| = L_w^{\a}$ and $F_m|I = f_{w,u}\circ \psi_I$ where $\psi_I$ is $(s/\a)$-H\"older with $\Hold_{s/\a}{\psi_I} =1$. Conversely, for any  $w\in A^*(r^{m-1})$ and $u\in A^*_w(r^m)$ there exists $I\in \mathscr{B}_m$ as above. Finally, $F_{m+1}|I = F_m|I$ for all $I\in\mathscr{B}_m$.
\end{enumerate}
\end{lem}

We now complete the proof of Theorem \ref{thm:main}, assuming Lemma \ref{lem:tree-param}. Let $\mathscr{B}_m$, $\mathscr{N}_m$, $\eta_m$ and $F_m$ be as in Lemma \ref{lem:tree-param}. Notice by (P2) that if $I\in\mathscr{N}_m$, then for all $F_n(I) \subset K_{\eta_m(I)}$. We claim that
\begin{equation}\label{eq:1'}
|F_m(x) - F_{m+1}(x)|_{\infty} \leq 2r^m.
\end{equation}
Equation (\ref{eq:1'}) is clear by (P5) if $x\in \mathscr{B}_m$. If $x\in \mathscr{N}_m$, then by (P2) and (P4) there exists $w\in A^*(r^m)$ such that $F_m(I)$ is an element of $K_w$ and $F_m(I) \subset F_{m+1}(I) \subset K_w$. Therefore,
\[ |F_m(x) - F_{m+1}(x)|_{\infty} \leq 2 \diam{K_w} < 2r^m.\]

We now claim that for all $m\in\N$ and all $x,y\in[0,1]$,
\begin{equation}\label{eq:2'}
|F_m(x) - F_{m}(y)|_{\infty} \lesssim_{L_1,s,\a} r^{m(1-\a/s)}|x-y|^{1/s}.
\end{equation}
To prove (\ref{eq:2'}) fix $x,y \in [0,M]$ and consider the following cases.

\emph{Case 1.} Suppose that there exists $I\in \mathscr{B}_m\cup\mathscr{N}_m$ such that $x,y \in I$. If $I\in\mathscr{N}_m$, (\ref{eq:2'}) is immediate since $F_m|I$ is constant. If $I\in\mathscr{B}_m$, then by (P5)
\[ |F_m(x) - F_{m}(y)|_{\infty} \lesssim_{L_1,s} \frac{\diam{f_m(I)}}{|I|^{1/s}}|x-y|^{1/s} = r^{m(1-\a/s)}|x-y|^{1/s}.\]

\emph{Case 2.} Suppose that there exist $I_1,I_2\in \mathscr{B}_m\cup\mathscr{N}_m$ such that $I_1\cap I_2$ is a single point $\{z\}$, $x\in I_1$ and $y\in I_2$. Then, by triangle inequality and Case 1,
\[ |F_m(x) - F_{m}(y)|_{\infty} \leq  |F_m(x) - F_{m}(z)|_{\infty} + |F_m(z) - F_{m}(y)|_{\infty} \lesssim_{L_1,s} 2r^{m(1-\a/s)}|x-y|^{1/s}.\]

\emph{Case 3.} Suppose that Case 1 and Case 2 do not hold. Let $m_0$ be the smallest positive integer $m$ such that there exists $I\in \mathscr{B}_m \cup\mathscr{N}_m$ with $x\leq z\leq y$ for all $z\in I$. In particular, suppose that
\[ a_1 \leq x \leq a_2 < a_3 < \cdots < a_n \leq y < a_{n+1},\]
where $[a_i,a_{i+1}] \in \mathscr{B}_{m_0}\cup\mathscr{N}_{m_0}$ for all $i\in\{1,\dots,n\}$. By minimality of $m_0$ and (P2), $n \leq 2L_1^{-s}r^{-s}$. By (P4) and (P5), $|a_i-a_{i+1}| \gtrsim_{L_1,s,\a} r^{\a m_0}$ and $F_m(a_i) = F_{m_0}(a_i)$ for all $i$. Furthermore, by (P2), (P3) and (P5) we have
\[ \max\{|F_m(x)-F_m(a_2)|_{\infty}, |F_m(y)-F_m(a_n)|_{\infty}\} \leq r^{m_0}.\]
Therefore, by Case 1 and the triangle inequality,
\begin{align*}
|F_m(x) &- F_m(y)|_{\infty} \\ &\leq |F_m(x)-F_m(a_2)|_{\infty} + \sum_{i=2}^{n-1} |F_m(a_i)-F_m(a_{i+1})|_{\infty} +  |F_m(y)-F_m(a_n)|_{\infty}\\
&\lesssim_{L_1,s} 2 r^{m_0} + r^{m_0(1-\a/s)}\sum_{i=2}^{n-1} |a_i-a_{i+1}|^{1/s} \\
&\lesssim_{L_1,s} r^{m_0(1-\a/s)}\sum_{i=2}^{n-1} |a_i-a_{i+1}|^{1/s}\\&\lesssim_{L_1,s}r^{m_0(1-\a/s)}\left(\sum_{i=2}^{n-1} |a_i-a_{i+1}|\right)^{1/s}\leq r^{m_0(1-\a/s)}|x-y|^{1/s}.
\end{align*}

By (\ref{eq:1'}), (\ref{eq:2'}) and Lemma \ref{l:LipHold}, we have that $F_m$ converges pointwise to a $(1/\a)$-H\"older continuous $F:[0,M] \to K$ with $\Hold_{1/\a}(F)\lesssim_{L_1,s,\a,M,r} 1\simeq_{L_1,s,\a}1$. By (P1), we have that $F([0,M]) \subset K$ and that $\bigcup_{m\in\N}\G_m \subset F([0,1])$. Therefore, $F([0,M]) = K$. This completes the proof of Theorem \ref{thm:main}, assuming Lemma \ref{lem:tree-param}.
\end{proof}

\begin{proof}[{Proof of Lemma \ref{lem:tree-param}}]
We give the construction of $\mathscr{B}_m$, $\mathscr{N}_m$, $\eta_m$ and $F_m$ in an inductive manner.

Suppose that $A^*(r) = \{w_1,\dots,w_n\}$. Decompose $[0,M]$ as
\[ [0,M] = I_1\cup J_1\cup I_1'\cup \cdots \cup I_n\cup J_n \cup I_n',\]
a union of closed intervals with mutually disjoint interiors, enumerated according to the orientation of $[0,M]$ such that $|I_j| = |I_j'| = 1$ and $|J_j| = M_{w_j}$. Set $\mathscr{B}_1 = \{I_1,I_1',\dots,I_n,I_n'\}$, $\mathscr{N}_1 = \{J_1,\dots,J_n\}$ and $\eta_1(J_j) = w_j$.

We now define $F_1: [0,M] \to \G_1$ as follows. For each $J_i\in \mathscr{N}_1$ let $F_{1}|J_i \equiv q_{w_i}$. For each $i \in \{1,\dots,n\}$, let $\psi_i: I_i \to [0,1]$ (resp. $\psi_i': I_i' \to [0,1]$) be a $(s/\a)$-H\"older orientation preserving (resp. orientation reversing) homeomorphism with $\Hold_{s/\a}{\psi_i} =1$ (resp. $\Hold_{s/\a}{\psi_i'} =1$). Define now $F_1|I_i = f_{\epsilon,w_i} \circ \psi_i$ and $F_1|I_i' = f_{\epsilon,w_i} \circ \psi_i'$. The properties (P1)--(P5) are easy to check.

Suppose now that for some $m\geq 1$, we have constructed $\mathscr{B}_m$, $\mathscr{N}_m$, $\eta_m$ and $F_m$ satisfying (P1)--(P5). For each $I\in \mathscr{B}_m$ define $F_{m+1}|I = F_m|I$. For each $I\in \mathscr{N}_m$ we construct families $\mathscr{B}_{m+1}(I)$ and $\mathscr{N}_{m+1}(I)$ and then we set
\begin{align*}
\mathscr{B}_{m+1} = \mathscr{B}_{m}\cup\bigcup_{I \in \mathscr{N}_m}\mathscr{B}_{m+1}(I),\qquad
\mathscr{N}_{m+1} = \bigcup_{I \in \mathscr{N}_m}\mathscr{N}_{m+1}(I).
\end{align*}
In the process we also define $F_{m+1}$ and $\eta_m$.

Suppose that $I\in \mathscr{N}_m$ and write $I=[a,b]$. By the inductive hypothesis (P3), there exists $w\in A^*(r^m)$ such that $F_m(I) = q_w$. Suppose that $A^*_w(r) = \{w_1,\dots,w_n\}$. Decompose $I$ as
\[ I =  I_1\cup J_1\cup I_1'\cup \cdots \cup I_l\cup J_l \cup I_l',\] a union
of closed intervals with mutually disjoint interiors, enumerated according to the orientation of $I$ such that $|I_j| = L_{w}^{\a}$ and $|J_j| = M_{w_j}$. Set $\mathscr{B}_{m+1}(I) = \{I_1,I_1',\dots,I_l,I_l'\}$, $\mathscr{N}_{m+1}(I) = \{J_1,\dots,J_l\}$ and $\eta_{m+1}|A^*_w(r^{m+1})(J_i) = w_i$.

For each $J_i\in \mathscr{N}_{m+1}(I)$ let $F_{m+1}|J_i \equiv q_{w_i}$. For each $i \in \{1,\dots,l\}$, let $\psi_i: I_i \to [0,L_w^s]$ (resp. $\psi_i': I_i' \to [0,L_w^s]$) be a $(s/\a)$-H\"older orientation preserving (resp. orientation reversing) homeomorphism with $\Hold_{s/\a}{\psi_i} =1$ (resp. $\Hold_{s/\a}{\psi_i'} =1$). Define now $F_{m+1}|I_i = f_{w,w_i} \circ \psi_i$ and $F_1|I_i' = f_{w,w_i} \circ \psi_i'$. The properties (P1)--(P5) are easy to check and are left to the reader.
\end{proof}

\section{H\"older parameterization of IFS without branching by arcs}\label{IFSsnowflakes}

On the way to the proof of Theorem \ref{thm:remes} (see \S\ref{sec:proof}), we first parameterize IFS attractors without branching by $(1/s)$-H\"older arcs (see \S\ref{sec:nobranching}), where $s$ is the similarity dimension. We then discuss how under the assumption of bounded turning, self-similar sets without branching are $(1/s)$-bi-H\"older arcs (see \S\ref{sec:bt}).
Finally, we give a family of examples of self-affine snowflake curves in the plane, for which the H\"older exponents in Theorem \ref{thm:connect} and Proposition \ref{prop:IFSarc} are sharp and may exceed 2 (see \S\ref{sec:sharpsnow}).

\subsection{IFS without branching} \label{sec:nobranching}

Given an IFS $\mathcal{F} = \{\phi_i : i \in A\}$ over a complete metric space, we say that $\mathcal{F}$ has \emph{no branching} or is \emph{without branching} if for every $m\in\N$ and word $w\in A^m$ (see \S\ref{sec:words}), there exist at most two words $u\in A^m\setminus\{w\}$ such that $\phi_w(K_\mathcal{F})\cap \phi_u(K_\mathcal{F}) \neq \emptyset$.

\begin{prop}[parameterization of connected IFS without branching]\label{prop:IFSarc} Let $\mathcal{F}$ be an IFS over a complete metric space; let $s=\sdim(\mathcal{F})$.
If $K_\mathcal{F}$ is connected, $\diam{K_{\mathcal{F}}}>0$, and $\mathcal{F}$ has no branching, then there exists a $(1/s)$-H\"older homeomorphism $f:[0,1] \to K$ with $\Hold_{1/s}{f} \lesssim_{L_1,s} \diam{K}$, where $L_1=\min_{\phi\in\mathcal{F}} \Lip\phi$.
\end{prop}

For the rest of \S\ref{sec:nobranching}, fix an IFS $\mathcal{F}=\{\phi_1,\dots,\phi_k\}$ over a complete metric space $(X,d)$ whose attractor $K:=K_\mathcal{F}$ is connected and has positive diameter. Adopt the notation and conventions set in the first paragraph of \S\ref{sec:Holdcon} as well as in \S\ref{sec:words}. In addition, assume that $\mathcal{F}$ has no branching. Since $\diam K>0$, $k\geq 2$. Replacing $\mathcal{F}$ with the iterated IFS $\mathcal{F}' = \{\phi_{w} : w\in A^2\}$ if needed, we may assume without loss generality that $k\geq 4$. Finally, rescaling the metric $d$, we may assume without loss of generality that $\diam K=1$.

Given $n\in\N$, we denote by $G_n$ the graph with vertices the set $A^n$ and (undirected) edges $\{\{w,u\} : w\neq u\text{ and }K_{w}\cap K_u \neq \emptyset\}$. For each $n\in\N$ and $w\in A^n$, the \emph{valence} $\val(u,G_n)$ of $w$ in $G_n$ is the number of all edges of $G_n$ containing $w$.

\begin{lem}\label{lem:arc}
Each $G_n$ is a combinatorial arc. Moreover, there exist exactly two distinct $i_0,j_0 \in A$ such that for any $n\in\N$ the following properties hold.
\begin{enumerate}
\item If $w$ has valence 1 in $G_n$, then there exists unique $i\in \{i_0,j_0\}$ such that $wi$ has valence 1 in $G_{n+1}$.
\item If $\{w,u\}$ is an edge of $G_n$, then there exist unique $i,j\in \{i_0,j_0\}$ such that $\{wi,uj\}$ is an edge in $G_{n+1}$.
\end{enumerate}
\end{lem}

\begin{proof} Because $\mathcal{F}$ has no branching, $\val(i,G_1)\in\{1,2\}$ for all $i\in A^1$. Therefore, either $G_1$ is a combinatorial circle or $G_1$ is a combinatorial arc. If $G_1$ is  a combinatorial circle and $\{1,i\}$ is any edge in $G_1$, then there exists $i_1,i_2,i_3,j_1 \in A$ such that $\{1i_1,1i_2\}$, $\{1i_2,1i_3\}$, and $\{1i_2,ij_1\}$ are edges in $G_2$; this implies $\val(1i_2,G_2) \geq 3$ and we reach a contradiction. Thus, in fact, $G_1$ is a combinatorial arc. In particular, there exist exactly two words in $A$ whose valence in $G_1$ is 1, say $i_0$ and $j_0$. The rest of the proof follows from a simple induction, which we leave the reader.
\end{proof}

From Lemma \ref{lem:arc}, we obtain two simple corollaries.

\begin{lem}\label{lem:intersection}
For all $n\in\N$ and all $w,u \in A^n$, $K_{w}\cap K_{u}$ is at most a point.
\end{lem}

\begin{proof}
Fix $w,u\in A^n$ such that $K_{w}\cap K_{u}\neq \emptyset$. We first claim that there exists unique $i\in A$ and unique $j\in A$ such that $K_{wi}\cap K_{uj}\neq \emptyset$. Assuming the claim to be true, we have
\[ \diam(K_{w}\cap K_{u}) = \diam(K_{wi}\cap K_{uj}) \leq L_k \diam(K_{w}\cap K_{u}) < \diam(K_{w}\cap K_{u}), \]
which implies that $\diam(K_{w}\cap K_{u}) = 0$.

To prove the claim, fix $i\in A$ such that $K_{wi}\cap K_{u}\neq \emptyset$. By Lemma \ref{lem:arc}, we have that $i\in\{i_0,j_0\}$ where $\{i_0,j_0\}$ are the unique elements of $A$ with valence 1 in $G_1$; say $i=i_0$. If there exists $w'\in A\setminus \{w,u\}$ such that $K_{w'}\cap K_{w}\neq \emptyset$, then by Lemma \ref{lem:arc} $K_{w'}\cap K_{wj_0}\neq \emptyset$ and $K_{wj_0}\cap K_{u} =\emptyset$. If no such $w'$ exists, then $\val(w,G_n)=1$ which implies that $\val(wj_0,G_{n+1})=1$ which also implies $K_{wj_0}\cap K_{u} =\emptyset$. In either case, $K_{wj_0}\cap K_{u} =\emptyset$ and $i$ is unique.
\end{proof}

\begin{lem}\label{lem:valency}
For all $n\in\N$, there exist exactly two words $w\in A^n$ such that the set $K_w\cap \overline{K\setminus K_w}$ contains only one point.
\end{lem}

\begin{proof}
By Lemma \ref{lem:arc}, for each $n\in\N$, there exist exactly two distinct words $w,u \in A^n$ that have valence $1$ in $G_n$. Fix one such word, say $w$. Then there exists a unique $w'\in A^n\setminus\{w\}$ such that $K_w\cap K_{w'} = K_w \cap \overline{K\setminus K_w}$. By Lemma \ref{lem:intersection}, the latter intersection is a single point.
\end{proof}

We are ready to prove Proposition \ref{prop:IFSarc}.

\begin{proof}[{Proof of Proposition \ref{prop:IFSarc}}]
By Lemma \ref{lem:valency}, there exist two infinite words $w_0,w_1 \in A^{\N}$ such that for all $n\in\N$, $w_0(n)$ and $w_1(n)$ are the unique vertices of valence 1 in $G_n$. Set
\[ \{v_0\} = \bigcap_{n=1}^{\infty} K_{w_0(n)} \qquad\text{and}\qquad \{v_1\} = \bigcap_{n=1}^{\infty} K_{w_1(n)}.\]

Fix $r\in(0,1)$ and let $f:[0,1] \to K$ be the map given by Proposition \ref{lem:Holderconn} with $x=v_0$ and $y=v_1$.
We already have that $f([0,1]) \subset K$. We claim that for all $m\in\N$ and all $w\in A^*(r^m)$,  we have $f([0,1]) \cap K_w \neq \emptyset$. Assuming the claim, it follows that $\dist(x,f([0,1])) \leq r^m$ for all $x\in K$ and all $m\in\N$. Hence $K\subset f([0,1])$ and $K=f([0,1])$.

Let $N = \max\{n\in\N : A^*(r^m)\cap A^n\neq \emptyset\}$. To prove the claim fix $w\in A^*(r^m)$. By Lemma \ref{lem:partition}, there exists $u\in A^N$ such that $K_u \subset K_w$. If $u\in\{w_0(N),w_1(N)\}$, then $K_w$ contains one of $v_0,v_1$, so $f([0,1]) \cap K_w \neq \emptyset$. If $u\not\in\{w_0(N),w_1(N)\}$, then $\val(u,G_N)=2$ and by Lemma \ref{lem:arc}, $K\setminus K_u$ has two components, one containing $v_0$ and the other containing $v_1$. Since $f([0,1])$ is connected and contains $v_0,v_1$, $\emptyset \neq f([0,1])\cap K_u \subset f([0,1])\cap K_w$.

It remains to show that $f$ is a homeomorphism and suffices to show that $f$ is injective. Recall the definitions of $\mathscr{E}_m$ and $f_m$ from the proof of Proposition \ref{lem:Holderconn}. By (P2) and (P3) therein, for each $m\in\N$ and $I\in\mathscr{E}_m$, there exists $w_I\in A^*(r^m)$ such that $f(I)\subset K_{w_I}$. Moreover, $w_I\neq w_J$ if $I\neq J$. In conjunction with the fact that $f([0,1])=K$, we have that $f(I) =K_{w_I}$. By design of the map $f$, it is easy to see that $K_{w_I}\cap K_{w_J}$ if and only if $I\cap J$. Assume $x,y \in [0,1]$ with $x\neq y$. Then there exists $m\in\N$ and disjoint $I,J \in \mathscr{E}_m$ such that $x\in I$ and $y\in J$. Hence $K_{w_I}\cap K_{w_J} = \emptyset$. Therefore, $f(I)\cap f(J) = \emptyset$, which yields $f(x)\neq f(y)$.
\end{proof}

\subsection{Bounded turning and self-similar bi-H\"older arcs} \label{sec:bt}
With additional information on the contractions of $\mathcal{F}$ and how the components $K_i=\phi_i(K)$ of the attractor $K$ intersect, the map $f$ constructed in Proposition \ref{prop:IFSarc} is actually a $(1/s)$-bi-H\"older homeomorphism. We say that $K$ has \emph{bounded turning} if there exists $C\geq 1$ such that for all distinct $i,j \in A$ with $K_i\cap K_j \neq \emptyset$: if $x\in K_i$, $y\in K_j$ and $z\in K_i\cap K_j$, then
\begin{equation}\label{eq:BT}
d(x,y)\geq C^{-1}\max\{d(x,z),d(y,z)\}.
\end{equation}
In general, self-similar curves (even in $\R^2$) do not have the bounded turning property; see \cite[Example 2.3]{ATK} or \cite[Theorem 2]{WenXi}.

The following proposition follows from Theorem 1.5 and Lemma 3.2 of  \cite{Iseli-Wildrick}.

\begin{prop}[self-similar sets without branching and with bounded turning] \label{prop:biHolder} Let $\mathcal{F}$ be an IFS over a complete metric space that is generated by similarities; let $s=\sdim(\mathcal{F})$. If $K_\mathcal{F}$ is connected, $\diam K_\mathcal{F}>0$, $\mathcal{F}$ has no branching, and $K_\mathcal{F}$ is bounded turning, then there exists a $(1/s)$-bi-H\"older homeomorphism $f:[0,1] \to K$.
\end{prop}

\subsection{Sharp exponents for self-affine snowflake curves in the plane}\label{sec:sharpsnow} For each line segment $l \subset \R^2$ and $\a\in(0,1)$, define the \emph{diamond $\mathcal{D}_\alpha(l)$ with axis $l$ and aperture $\a$},
\[ \mathcal{D}_{\a}(l) := \{x\in\R^2 : \dist(x,l) \leq \a\min(|x-p|,|x-q|)\},\]
where $p,q$ are the endpoints of $l$. We will build a family of self-affine snowflake curves as the IFS attractor of a chain of diamonds. Let $l_0:=[0,1]\times\{0\}$ and let $P = l_1\cup\cdots\cup l_k$, $k\geq 2$, be a polygonal arc lying in $\{0,1\}\cup\interior\mathcal{D}_{1/2}(l_0)$, enumerated so that
\begin{itemize}
\item $l_i\cap l_j \neq \emptyset$ if and only if $|i-j|\leq 1$,
\item $(0,0)$ is an endpoint of $l_1$ and $(1,0)$ is an endpoint of $l_k$.
\end{itemize}
Choose apertures $\a_i \in (0,1/2)$ small enough so that \begin{equation}\label{e:disjoint-diamonds}\mathcal{D}_{\a_i}(l_i)\cap \mathcal{D}_{\a_j}(l_j) = l_i\cap l_j\quad\text{for all }1\leq i< j\leq k.\end{equation} For each $i\in\{1,\dots,k\}$, fix an affine homeomorphism $\phi_i:\R^2 \to \R^2$ such that $\phi_i(l_0)=l_i$ and $\phi_i(\mathcal{D}_{1/2}(l_0)) = \mathcal{D}_{\a_i}(l_i)$. Because each aperture $\alpha_i<1/2$, \begin{equation*} \Lip(\phi_i) = |l_i|<1\quad\text{for all }1\leq i\leq k,\end{equation*} where $|l_i|$ denotes the length of $l_i$. In particular, $\mathcal{F}=\{\phi_i:1\leq i\leq k\}$ is an IFS over $\R^2$; see Figure \ref{diamond-generators}. \begin{figure}[thb]
\begin{center}\includegraphics[width=3in]{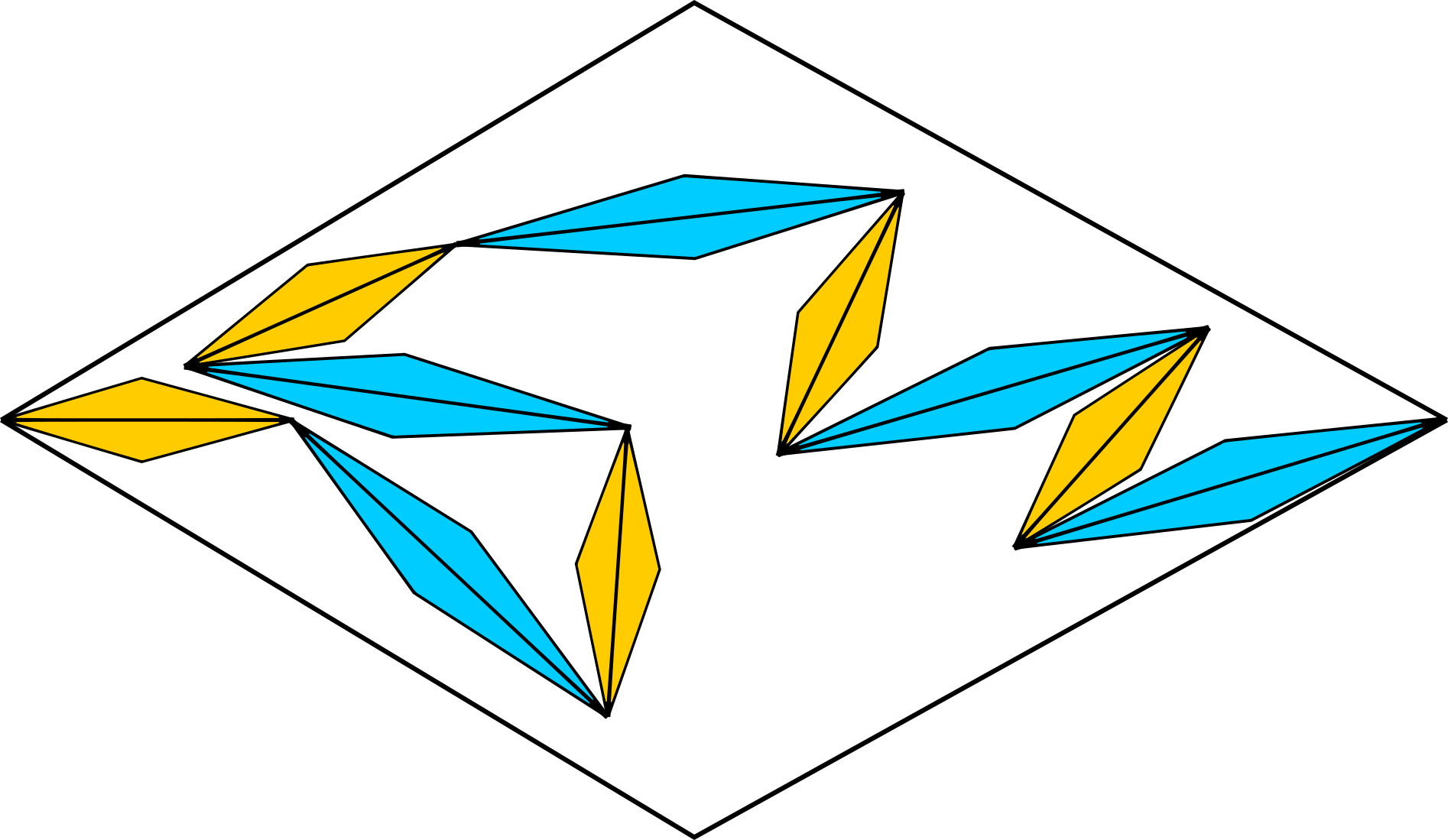}\end{center}
\caption{Generators of an IFS for a self-affine snowflake curve.}\label{diamond-generators}
\end{figure} Let $s=\sdim(\mathcal{F})$ and let $K=K_\mathcal{F}$ denote the attractor of $\mathcal{F}$. Since $\mathcal{F}$ has no branching, the snowflake curve $K$ is a $(1/s)$-H\"older arc by Proposition \ref{prop:IFSarc}; the endpoints of $K$ are $p_0 = (0,0)$ and $p_1=(1,0)$. We now show that the exponent cannot be increased.

\begin{lem}
If $p_0$, $p_1$ are connected by a $(1/\a)$-H\"older curve in $K$, then $\a\geq \sdim(\mathcal{F})$.
\end{lem}

\begin{proof}
Fix a $(1/\a)$-H\"older map $f:[0,1] \to K$ such that $f(0)=p_0$ and $f(1)=p_1$ and write $H:=\Hold_{1/\a}(f)$. Since $K$ has positive diameter, $H>0$. Let $A=\{1,\dots,k\}$ denote the alphabet associated to $\mathcal{F}$. Fix a generation $n\in\N$, and for each $w\in A^n$, choose an interval $I_w \subset [0,1]$ such that $f(I_w)=\phi_w(K)$. The intervals $\{I_w:w\in A^n\}$ have mutually disjoint interiors by \eqref{e:disjoint-diamonds}. Thus,
\[ 1 \geq \sum_{w\in A^n}|I_w| \geq H^{-\a} \sum_{w\in A^n}(\diam{\phi_w(K)})^{\a} = H^{-\a} \sum_{w\in A^n}|\phi_w(l_0)|^{\a} = H^{-\a} \left(\sum_{i\in A}|l_i|^{\a}\right)^n.\]
Since $n$ can be arbitrarily large, $\sum_{i\in A}|l_i|^{\a} \leq 1$. Therefore, $\a\geq \sdim(\mathcal{F})$.
\end{proof}

As a final remark, we note that it is possible to choose $P$ so that $|l_1|^2 + \cdots+|l_k|^2 > 1$, in which case $\sdim(\mathcal{F})>2$. In particular, there exist self-affine snowflake curves $\Gamma\subset\R^2$ such that $\Gamma$ is a $(1/\alpha)$-H\"older curve if and only if $\alpha\geq \alpha_0(\Gamma)>2$.

\section{H\"older parameterization of self-similar sets (Remes' method)}\label{sec:proof}

Our goal in this section is to record a proof of Theorem \ref{thm:remes} that combines original ideas of Remes \cite{Remes} with our style of H\"older parameterization from above. To aid the reader wishing to learn the proof, we have attempted to include  a clear description of the key properties of parameterizations that approximate the final map (see Lemma \ref{prelim}), which are obscured in Remes' thesis.

Fix an IFS $\mathcal{F}=\{\phi_1,\dots,\phi_k\}$ over a complete metric space $(X,d)$; let $s=\sdim(\mathcal{F})$. Assume that $\mathcal{F}$ is generated by similarities, $K=K_{\mathcal{F}}$ is connected, $\diam K>0$, and $\mathcal{H}^s(K)>0$, where $s=\sdim(\mathcal{F})$. Recall that $\mathcal{H}^s(K)>0$ implies $\mathcal{F}$ satisfies the strong open set condition by Theorem \ref{thm:Schief}. Moreover, by Lemma \ref{lem:sdimvsdim}, $K$ is Ahlfors $s$-regular; thus, we can find constants $0<C_1 \leq C_2 <\infty$ such that
\begin{equation}\label{e:sreg} C_1 \rho^s \leq \mathcal{H}^s(K\cap B(x,\rho)) \leq C_2 \rho^s \quad\text{for all $x\in K$ and all $0< \rho \leq \diam K$}.\end{equation} As usual, we adopt the notation and conventions set in the first paragraph of \S\ref{sec:Holdcon} as well as in \S\ref{sec:words}. Rescaling the metric, we may assume without loss of generality that $\diam K=1$. Since $K$ is self-similar, it follows that \begin{equation} \diam K_w = L_w\quad\text{for all }w\in A^*,\end{equation} \begin{equation}\label{e:ss-diam} L_1 \delta \leq \diam K_w < \delta\quad\text{for all }w\in A^*(\delta).\end{equation} If $\mathcal{F}$ has no branching (see \S\ref{sec:nobranching}
), then a $(1/s)$-H\"older parameterization of $K$ already exists by Proposition \ref{prop:IFSarc}
. Thus, we shall assume $\mathcal{F}$ has \emph{branching}, i.e.~there exists $m\in\N$ and distinct words $w_1,\dots,w_4\in A^m$ such that $K_{w_1}\cap K_{w_i} \neq \emptyset$ for each $i\in\{2,3,4\}$. In the event that $m\geq 2$ (see Example \ref{ex:sierpinski}), we replace $\mathcal{F}$ with the self-similar IFS $\mathcal{F}' = \{\phi_w : w\in A^m\}$. This causes no harm to the proof, because the attractors coincide, i.e.~ $K_\mathcal{F'}=K_\mathcal{F}$, and $\sdim(\mathcal{F'})=\sdim(\mathcal{F})$. Therefore, without loss of generality, we may assume that there exist distinct letters $i_1,i_2,i_3,i_4 \in A$ such that \begin{equation}\label{e:branching} K_{i_1}\cap K_{i_j} \neq \emptyset\quad\text{for each }j\in\{2,3,4\}.\end{equation}

\begin{ex}\label{ex:sierpinski} Divide the unit square into $3\times 3$ congruent subsquares with disjoint interiors $S_i$ ($1\leq i\leq 9$). Let $S_9$ denote the central square and for each $1\leq i\leq 8$, let $\psi_i:\R^2\to \R^2$ be the unique rotation-free and reflection-free similarity that maps $[0,1]^2$ onto $S_i$. The attractor of the IFS $\mathcal{G}=\{\psi_1,\dots,\psi_8\}$ is the \emph{Sierpi\'nski carpet}. Looking only at the intersection pattern of the first iterates $\psi_1(K_\mathcal{G})$,\dots,$\psi_8(K_\mathcal{G})$, it appears that $\mathcal{G}$ has no branching. However, upon examining the intersections of the second iterates $\psi_i\circ \psi_j(K_\mathcal{G})$ ($1\leq i,j\leq 8$), it becomes apparent that $\mathcal{G}$ has branching. \end{ex}

To continue, use the Kuratowski embedding theorem to embed $(K,d)$ into $(\ell_{\infty},|\cdot|_\infty)$. (If $K$ already lies in some Euclidean or Banach space, or in a complete quasiconvex metric space, then the construction below can be carried out in that space instead.) Let $d_H$ denote the Hausdorff distance between compact sets in $\ell_\infty$. By the Arzel\'a-Ascoli theorem, to complete the proof of Theorem \ref{thm:remes}, it suffices to establish the following claim.

\begin{prop}\label{prop:Remes}
There exists a sequence $(F_N)_{N=1}^\infty$ of $(1/s)$-H\"older continuous maps $F_N :[0,1] \to \ell_{\infty}$ with uniformly bounded H\"older constants such that
\[ \lim_{N\to \infty}d_H(F_N([0,1]),K)= 0.\] \end{prop}

\begin{rem}\label{r:aa} It is perhaps unfortunate that we have to invoke the Arzel\'a-Ascoli theorem to implement Remes' method. We leave as an \emph{open problem} to find a proof of Theorem \ref{thm:remes} that avoids taking a subsequential limit of the intermediate maps; cf.~ the proofs in \S3 above or the proof of the H\"older traveling salesman theorem in \cite{BNV}.\end{rem}

We devote the remainder of this section to proving Proposition \ref{prop:Remes}.

\subsection{Start of the Proof of Proposition \ref{prop:Remes}}

To start, since $\mathcal{F}$ satisfies the strong open set condition, there exists an open set $U \subset X$ such that $U\cap K \neq \emptyset$, $\phi_i(U)\subset U$ for all $i\in A$ and $\phi_i(U) \cap\phi_j(U) =\emptyset$ for all $i,j \in A$ with $i\neq j$. Fix a point $v\in U \cap K$, choose $\tau\in(0,1/2)$ such that $B_X(v,\tau)\subset U$, and assign $r:= \frac14L_1\tau$. Then, since $\mathcal{F}$ consists of similarities,
\begin{equation}\label{e:separate}
|\phi_w(v)-\phi_u(v)|_\infty\geq  (L_w+L_u)\tau \geq 2L_1\tau r^m = (8r)r^m\quad\text{for all distinct }w,u\in A^*(r^m),
\end{equation}
because the balls $\phi_w(B(v,\tau))=B(\phi_w(v),L_w \tau)$ and $\phi_u(B(v,\tau))=B(\phi_u(v),L_u\tau)$ in $X$ are disjoint. Indeed, if $w_0$ is the longest word in $A^*$ such that $K_w,K_u \subset K_{w_0}$, then for some distinct $i,j\in A$, $\phi_w(B(v,\tau)) \subset \phi_{w_0i}(B(v,\tau)) \subset \phi_{w_0i}(U)$ and $\phi_u(B(v,\tau)) \subset \phi_{w_0j}(B(v,\tau)) \subset \phi_{w_0j}(U)$.

For all $m\in\N$, define the set \begin{equation}\label{e:Y} Y_m:=\{\phi_w(v) : w \in A^*(r^m)\}.\end{equation} The separation condition \eqref{e:separate} ensures that the words in $A^*(r^m)$ and points in $Y_m$ are in one-to-one correspondence. Unfortunately, the sets $Y_m$ are not necessarily nested.

To proceed, fix an index $N\in\N$. We will construct a map $F_N:[0,1]\rightarrow \ell_\infty$ with $\Hold_{1/s} F_N \lesssim_{L_1,s,\tau,C_1,C_2} 1$ and $d_H(F_N([0,1]),K)\lesssim_{L_1,\tau} r^N$.

\subsection{Nets} Following an idea of Remes \cite{Remes}, starting from $Y_N$ and working backwards through $Y_1$, we now produce a nested sequence of sets $V_1\subset\dots\subset V_N$ recursively, as follows.
Set $V_N:=Y_N$. Next, assume we have defined $V_m,\dots,V_N$ for some $2\leq m\leq N$ so that
\begin{enumerate}
\item $ V_m \subset V_{m+1} \subset \cdots\subset V_N = Y_N$; and,
\item for each $i \in \{m,\dots,N\}$ and each $w\in A^*(r^{i})$, there exists a unique $x \in K_w\cap V_i$.
\end{enumerate}
Replace each $x\in Y_{m-1}$ by an element $x'\in V_m \cap K_{u_x}$ of shortest distance to $x$, where $u_x\in A^*(r^{m-1})$ satisfies $\phi_{u_x}(v) = x$. This produces the set $V_{m-1}$. See Figure \ref{fig:y-v}. \begin{figure}[t]
\begin{center}\includegraphics[width=.8\textwidth]{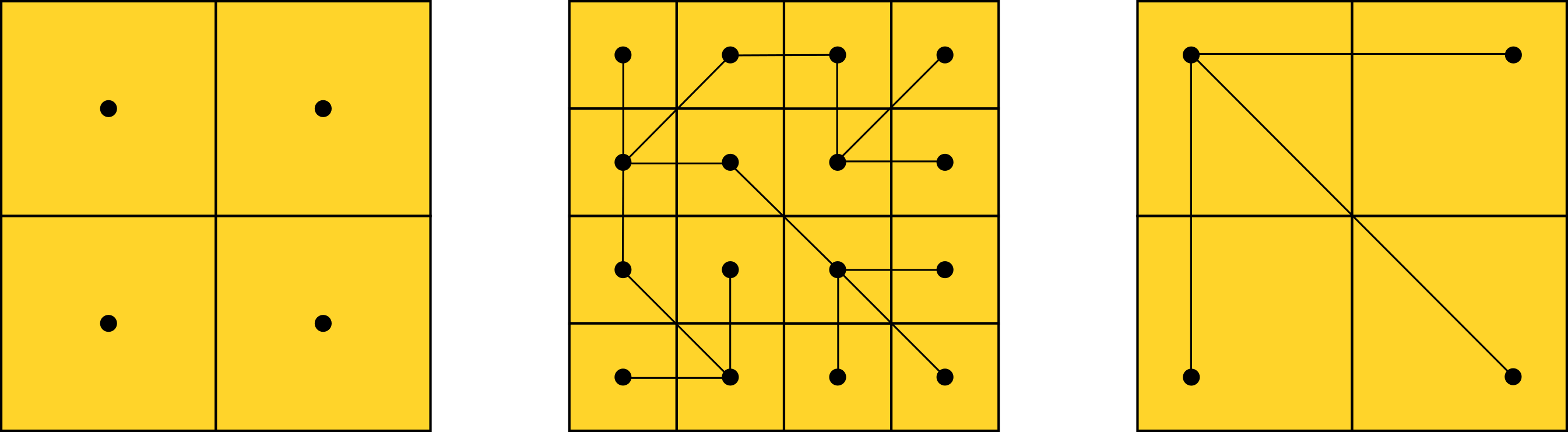}\end{center}
\caption{Schematic for points in the sets $Y_1$ (left), $Y_2=V_2$ (center), and $V_1$ (right) for a self-similar IFS for the square and generation $N=2$. Possible realizations of the trees $T_1$ (right) and $T_2$ (center).}\label{fig:y-v}
\end{figure}
\begin{rem}\label{rem:aa2}
The recursive definition of the sets $V_1\subset\dots\subset V_N$ starting from a fixed level $Y_N$ is one obstacle to proving Theorem \ref{thm:remes} without using the Arzel\'a-Ascoli theorem.
\end{rem}

\begin{lem}[properties of the sets $V_m$]\label{lem:Remmes1}
Let $m\in\{1,\dots,N\}$.
\begin{enumerate}
\item For each $w\in A^*(r^m)$, there exists a unique $x\in V_m\cap K_w$.
\item If $m\leq N-1$, then $V_m\subset V_{m+1}$ and for every $x\in V_{m+1}$ there exists $x'\in V_{m}$ such that $|x-x'|_\infty<r^m$.

\item If $w\in A^*(r^m)$ and $x\in V_m\cap K_w$, then $|x-\phi_w(v)|_{\infty} < (2r) r^{m}$.
\item For all distinct $a,b \in V_m$, we have $|a-b|_{\infty} >(4r) r^m$.
\end{enumerate}
\end{lem}

\begin{proof}

The first claim and nesting property $V_m\subset V_{m+1}$ follow immediately by the design of the sets $V_m$. Suppose that $1\leq m\leq N-1$ and $x\in V_{m+1}$. By (1), there exists $w\in A^*(r^{m+1})$ such that $x\in K_w$, say $w=i_1\dots i_n \in A^n$. Set $w'=i_1\dots i_{n-1}$. Then $L_w < r^{m+1} \leq L_{w'}\leq L_w/L_1$, since $w\in A^*(r^{m+1})$. If $L_w < r^{m} \leq L_{w'}$, as well, then $w\in A^*(r^m)$, $x\in V_m$, and we take $x'=x$. Otherwise, $L_{w'}<r^m$. Choose $w''=i_1\dots i_l$, $l\leq n-1$ to be the shortest word such that $L_{w''}<r^m$. Then $w''\in A^*(r^m)$. By (1), there exists a unique $x'\in V_m\cap K_{w''}$. Then $|x-x'|_\infty \leq \diam K_{w''}<r^m$ by \eqref{e:ss-diam}. This establishes the second claim.

For the third claim, we first prove that for all $w\in A^*(r^m)$ and $x\in V_m\cap K_w$, \begin{equation} \label{e:golf} |x-\phi_w(v)| \leq \left\{\begin{array}{ll} r^{m+1} + \dots + r^{N} &\text{ if }m\leq N-1,\\ 0 &\text{ if } m=N,\end{array}\right.\end{equation} by backwards induction on $m$. Equation \eqref{e:golf} holds in the base case, because $V_N=Y_N$. Suppose for induction that we have established \eqref{e:golf} for some $2\leq m+1\leq N$, and let $w\in A^*(r^m)$ and $x\in V_m\cap K_w$. There exists $wu\in A^{*}(r^{m+1})$ such that $\phi_{w}(v)\in K_{wu}$. Also, by (2), there exists $y\in V_{m+1}\cap K_{wu}$. On one hand, $|\phi_{wu}(v)-\phi_w(v)|_{\infty} \leq \diam K_{wu}<  r^{m+1}$ by \eqref{e:ss-diam}. On the other hand, by the induction hypothesis, $|y-\phi_{wu}(v)|_{\infty}\leq r^{m+2}+\dots+ r^{N}$. Thus, since $x$ is by definition a point in $V_{m+1}$ that is nearest to $\phi_w(v)$, $$|x-\phi_w(v)| \leq |y-\phi_w(v)| \leq  r^{m+1} + r^{m+2}+\dots+r^{N}.$$ Therefore, \eqref{e:golf} holds for all $m$. Claim (3) follows, because $$r^{m+1}+\dots+r^{N} = r^{m+1} (1-r^{N-m})/(1-r)<2r^{m+1},$$ where the last inequality holds since $r<1/2$.

Finally, for the last claim, if $a,b\in V_m$ are distinct, say with $a\in K_w\cap V_m$ and $b\in K_u \cap V_m$ for some $w,u\in A^*(r^m)$, then by \eqref{e:separate},
\begin{align*}
|a-b|_{\infty} &\geq |\phi_w(v)-\phi_u(v)|_{\infty} - |a-\phi_w(v)|_{\infty} - |b-\phi_u(v)|_{\infty}\\ &> (8r) r^m - 2(2r^{m+1})= (4r) r^m. \qedhere
\end{align*}
\end{proof}

\subsection{Trees}
Next, we define a finite sequence of trees $T_m=(V_m,E_m)_{m=1,\dots,N}$ inductively, where the vertices $V_m$ were defined in the previous section and the edges $E_m$ will be specified below. By Lemma \ref{lem:Remmes1}, for all $m\in\{1,\dots,N\}$ and all $x\in V_m$, there exists a unique $w\in A^*(r^m)$ such that $x\in K_w$; we denote this word $w$ by $x(m)$.

Let $G_1 = (V_1,\hat{E}_1)$ be the graph whose edge set is given by $$\hat{E}_1= \{\{x,y\} : x\neq y\text{ and } K_{x(1)}\cap K_{y(1)} \neq \emptyset\}.$$
The connectedness of $K$ implies that $G_1$ is a connected graph, but not necessarily a tree. Now, removing some edges from $\hat{E}_1$, we obtain a new set $E_1$ so that $T_1 = (V_1,E_1)$ is a connected tree. Because we assumed $\mathcal{F}$ has branching, see \eqref{e:branching}, we may assume that $T_1$ has at least one branch point, i.e.~there exists $x \in V_1$ with valence in $T_1$ at least 3.

Suppose that we have defined $T_m=(V_m,E_m)$ for some $m \in \{1,\dots,N-1\}$. For each $x\in V_m$, let $V_{m+1,x} = V_{m+1}\cap K_{x(m)}$ and let $T_{m+1,x} = (V_{m+1,x},E_{m+1,x})$ be a connected tree such that $\{y,z\} \in E_{m+1,x}$ only if $y,z \in V_{m+1,x}$, $y\neq z$ and $K_{y(m+1)}\cap K_{z(m+1)} \neq \emptyset$. Moreover, since $K_{x(m)}$ is homothetic to $K$, we may require that $T_{m+1,x}$ has at least one branch point. Now, if $\{a,b\}\in E_m$, there exists $a' \in V_{m+1,a}$ and $b'\in V_{m+1,b}$ such that $K_{a'(m+1)}\cap K_{b'(m+1)} \neq \emptyset$. There is not a canonical choice, so we select one pair $\{a',b'\}$ for each pair $\{a,b\}$ in an arbitrary fashion. Set
\[ E_{m+1} := \bigcup_{x\in V_m}E_{m+1,x} \cup \bigcup_{\{a,b\} \in E_m} \{\{a',b'\}\}.\] This completes the definition of the trees $T_1,\dots,T_N$. Below all trees $T_m$ are realized in $\ell_{\infty}$ through the natural identification of $\{a,b\}\in E_m$ with the line segment $[a,b]$.

\begin{lem}[length of edges] \label{l:short} For all $m\in\N$, the length $|x-y|_\infty$ of each edge $[x,y]$ in $T_m$ is at least $(8r)r^m$ and less than $2r^m$.\end{lem}

\begin{proof} By construction, for each edge $[x,y]$ in $T_m$, we have $K_{x(m)}\cap K_{y(m)}\neq\emptyset$. Hence $|x-y|_\infty \leq \diam K_{x(m)}+\diam K_{y(m)} < 2r^m$ by \eqref{e:ss-diam}. The lower bound on the length is taken from \eqref{e:separate}. \end{proof}

\subsection{Parameterization of $T_N$ and the map $F_N$}\label{sec:Remes2}
For each $1\leq m\leq N$, we denote by $T_{N,m}$ the minimal subgraph of $T_N$ that contains $V_m$. See Figure \ref{fig:subtree}. Clearly, $T_{N,m}$ is a connected subtree of $T_{N,n}$ whenever $m\leq n$ and $T_{N,N}=T_N$.

\begin{figure}[th]
\begin{center}\includegraphics[width=.23\textwidth]{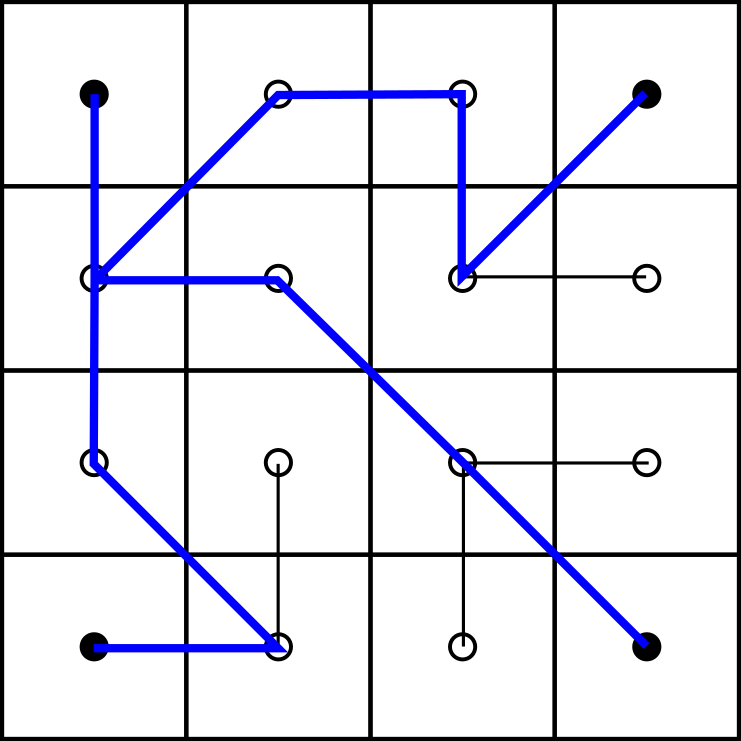}\end{center}
\caption{The tree $T_{2,1}$ (in blue) for the IFS for the square in Fig.~3.}\label{fig:subtree}\end{figure}

\begin{lem}[intermediate parameterizations] \label{prelim} There exists a constant $c>0$ depending only on $L_1,s,\tau,C_1,C_2$, and there exists a collection $\mathscr{E}_m$ of closed nondegenerate intervals in $[0,1]$ and a continuous map $f_m:[0,1]\rightarrow T_N$ for each $1\leq m\leq N$ with the following properties.
\begin{enumerate}
\item[{(P1)}] The intervals in $\mathscr{E}_m$ have mutually disjoint interiors and their union $\bigcup\mathscr{E}_m = [0,1]$.
\item[{(P2)}] The map $f_m$ is a 2-to-1 piecewise linear tour of edges of a subtree $\widetilde{T}_{N,m}$ of $T_N$ containing $T_{N,m}$.
\item[{(P3)}] For every $I=[a,b]\in\mathscr{E}_m$, we have the image of the endpoints $f_m(a),f_m(b)$ are vertices of $\widetilde{T}_{N,m}$.
\item[{(P4)}] For every $I=[a,b]\in\mathscr{E}_m$ and $x\in[a,b]$,
 \[ (2r)r^m\leq |f_m(a)-f_m(b)|_\infty \leq (4/r)r^m\quad\text{and}\quad |f_m(a)-f_m(x)|_\infty\leq (5/r)r^m.\]
\item[{(P5)}] For all $1\leq m\leq N-1$ and for every $I\in\mathscr{E}_m$, we have $f_{m+1}|\partial I=f_m|\partial I$ and $f_m(I) \subset f_{m+1}(I)$.
\item[{(P6)}] For all $1\leq m\leq N$ and $I\in\mathscr{E}_m$, $f_N|I$ tours at least $c r^{-(N-m)s}$ edges in $T_N$.
\item[{(P7)}] When $m=N$, $\widetilde{T}_{N,N}=T_{N,N}=T_N$ and $f_N(I)$ is an edge in $T_N$ for each $I\in\mathscr{E}_N$.
\end{enumerate}
\end{lem}

We now show how to use Lemma \ref{prelim} to construct a $(1/s)$-H\"older continuous surjection $F_N:[0,1]\rightarrow T_N$ with $\Hold_{1/s} F_N \lesssim_{L_1,s,\tau,C_1,C_2} 1$ and $d_H(F_N([0,1]),K)\lesssim_{L_1,\tau} r^N$, where $d_H$ is the Hausdorff distance in $\ell_\infty$. This reduces the proof of Proposition \ref{prop:Remes} to verification of Lemma \ref{prelim}.

First of all, by (P2) and (P7), $\card \mathscr{E}_N= 2(\card(V_N)-1)$. Let $\psi:[0,1] \to [0,1]$ be the unique continuous, nondecreasing function such that
$\psi|I$ is linear and $|\psi(I)|=(\card\mathscr{E}_N)^{-1}$ for all $I\in\mathscr{E}_N$. Let $F_N:[0,1] \to T_N$ be the unique map satisfying $f_N = F_N\circ \psi$ (i.e.~$F_N:=f_N\circ \psi^{-1}$). Thus, $F_N$ is a 2-to-1 piecewise linear tour of the edges of $T_N$ in the order determined by $f_N$, where the preimage of each edge has equal length. By (P2), (P7), the definition of the set $V_N$, and the fact that $|x-y|\leq 2r^N$ for any two adjacent vertices of $T_N$,
\begin{equation}
d_H(F_N([0,1]),K) = d_H(T_N,K) \leq d_H(T_N,V_N) + d_H(V_N,K) \leq 3r^N.
\end{equation}
It remains to show that $F_N$ is $(1/s)$-H\"older with H\"older constant independent of $N$.

To that purpose, we define an auxiliary sequence $F_N^1,\dots,F_N^N\equiv F_N$ to which we can apply Corollary \ref{c:LipHold}. As already noted, we simply set $F_N^N:=F_N$. Next, suppose that $1\leq m\leq N-1$. Let $\mathscr{N}_m=\{a_1,a_2,\dots,a_l\}$ denote the set of endpoints of intervals in $\mathscr{E}_m$, enumerated according to the orientation of $[0,1]$. Let $\tilde f_j:[0,1]\rightarrow\ell_\infty$ be defined by linear interpolation and the rule $\tilde f_j(a_i)=f_j(a_i)$ for all $i$. We then let $F_N^j$ be the unique map such that $\tilde f_j=F_N^j\circ \psi = f_j$ (i.e.~ $F_N^j:=\tilde f_j\circ \psi^{-1}$). By (P3), (P4) and (P5), for all $1\leq j\leq N-1$,
\begin{equation}\label{eq:RemHold1}
|F_N^j(x) - F_{N}^{j+1}(x)|_{\infty} \lesssim_{L_1,\tau} r^{j}\quad\text{for all }x\in[0,1].
\end{equation}

Next, we claim that for all $1\leq j\leq N$ and all $x,y\in [0,1]$,
\begin{equation}\label{eq:RemHold2}
|F_N^j(x) -F_{N}^{j}(y)|_{\infty} \lesssim_{L_1,s,\tau,C_1,C_2} r^{j(1-s)}|x-y|.
\end{equation} Since each map $F_N^j$ is continuous and linear on each interval $\psi(I)$, $I\in\mathscr{E}_j$, the Lipschitz constant is given by $$\Lip(F_N^j) = \max_{I\in\mathscr{E}_j} \frac{\diam F_N^j(\psi(I))}{|\psi(I)|}\lesssim_{L_1,\tau} \max_{I\in\mathscr{E}_j}\frac{r^j}{|\psi(I)|}$$ by (P3). Fix $I\in\mathscr{E}_j$. To estimate $|\psi(I)|$,
by (P6), (P7), and Lemma \ref{6} we have
\begin{equation*}|\psi(I)| = \frac{\card\{J \in \mathscr{E}_N : J\subset I \}}{2(\card(V_N)-1)}  \gtrsim_{L_1,s,\tau,C_1,C_2} \frac{r^{-(N-j)s}}{r^{-Ns}} \gtrsim_{L_1,s,\tau,C_1,C_2} r^{js}.\end{equation*}
Thus, we have established \eqref{eq:RemHold2}.

Therefore, by (\ref{eq:RemHold1}), (\ref{eq:RemHold2}), and Corollary \ref{c:LipHold}, $F_N\equiv F_N^N$ is a $(1/s)$-H\"older map with H\"older constant depending only on $L_1,s,\tau,C_1,C_2$. This completes the proof of Proposition \ref{prop:Remes} and Theorem \ref{thm:remes}, up to verifying Lemma \ref{prelim}.

\subsection{Remes' Branching Lemma and the Proof of Lemma \ref{prelim}}

We now recall a key lemma from Remes \cite{Remes}, which lets us build the intermediate parameterizations in Lemma \ref{prelim}. In the remainder of this section, we frequently use the following notation and terminology. Given $a,b \in V_m$ with $a\neq b$, we let $R_m(a,b)$ denote the unique arc (the ``road") in $T_m$ with endpoints $a$ and $b$. A \emph{branch $B$ of $T_m$ with respect to $R_m(a,b)$} is a maximal connected subtree of $T_m$ with at least two vertices such that $B$ contains precisely one vertex $x$ in $R_m(a,b)$ and $x$ is terminal in $B$ (i.e.~ $x$ has valency 1 in $B$). See Figure \ref{fig:branch}. More generally, if $T$ is a connected tree and $S$ is a connected subtree of $T$, we define a \emph{branch $B$ of $T$ with respect to $S$} to be a  maximal connected subtree of $T$ with at least two vertices such that $B$ contains precisely one vertex $x$ in $S$, and $x$ is terminal in $B$.

\begin{figure}[th]
\begin{center}\includegraphics[width=.23\textwidth]{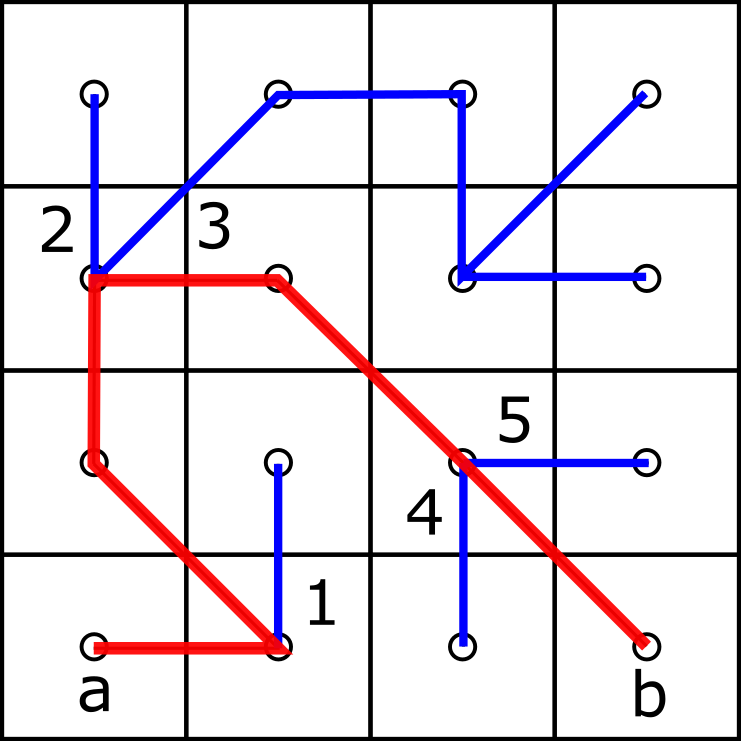}\end{center}
\caption{A road $R_2(a,b)$ (in red) and the 5 branches in $T_2$ with respect to $R_2(a,b)$ (in blue) for the IFS for the square in Fig.~3. Branches 2 and 3  contain a point in $V_1\setminus R_2(a,b)$; branches 1, 4, and 5 do not.}\label{fig:branch}
\end{figure}

\begin{lem}[Remes' branching lemma {\cite[Lemma 4.11]{Remes}}]\label{lem:Remes}
Let $a,b\in V_N$ with $a\neq b$ and let $R \subset V_N$ be the set of vertices of $R_N(a,b)$. Suppose that there exists $m\leq N$ such that $|a-b|_{\infty}\geq (2r) r^m$ and $|a-x|_{\infty} \leq (4/r)r^m$ for all $x\in R$.
\begin{enumerate}
\item \emph{Control on number of branches from above}: There exists $C\geq 1$ depending only on $L_1,s,\tau,C_1,C_2$ such that the number of the branches of $T_N$ with respect to $R_N(a,b)$ containing points in $V_m\setminus R$ is less than $C$.
\item \emph{Control of the road length}: There exists $C'\geq 1$ depending only on $L_1,s,\tau, C_1,C_2$ such that if $S=\{z_1,\dots,z_l\}$ is a subset of $R$, enumerated relative to the ordering induced by $R_N(a,b)$, and $|z_i-z_{i+1}|_\infty\geq (2r) r^m$ for all $1\leq i\leq l-1$, then $l\leq C'$.
\item \emph{Control on number of branches from below}: There is $t\in\N$ depending only on $L_1,\tau,s,C_1,C_2$ such that if $m\leq N-t$, then the number of branches of $T_N$ with respect to $R_N(a,b)$ that contain some vertex in $V_{m+t} \setminus R$ is at least $2C+C'+2$. Moreover, if $c\in V_{m+t}\setminus R$ is such a vertex and $c'\in K_{x_{t+m}(c)}$, then $c$ and $c'$ belong to the same branch of $T_N$ with respect to $R_N(a,b)$.
\end{enumerate}
\end{lem}

\begin{proof} From the inductive construction, it is easy to see that the trees $T_1,\dots,T_N$ satisfy the following property, which Remes calls the \emph{branch-preserving property}:

\begin{quotation}\cite[p.~23]{Remes} Let $1\leq m\leq n\leq N$, let $x_1,x_2 \in V_m$, let $B$ be a branch of $T_m$ with respect to $R_m(x_1,x_2)$, and let $x_3$ be a vertex of the branch $B$. Let $x_1',x_2' \in V_n$ with $x_1'\in K_{x_1(m)}$ and $x_2' \in K_{x_2(m)}$. Then all vertices in $V_n\cap K_{x_3(m)}$ belong to the same branch of $T_n$ with respect to $R_n(x_1',x_2')$.
\end{quotation}
Since we arranged for the attractor in our setting to satisfy \eqref{e:sreg}, the proof of Lemma \ref{lem:Remes} follows exactly as the proof of \cite[Lemma 4.11]{Remes} in Euclidean space. This is the only place in the proof of Theorem \ref{thm:remes} where we use the assumption that $\Haus^s(K)>0$.

(1) Denote by $\mathcal{B}$ the set of branches of $T_N$ with respect to $R_N(a,b)$ containing points in $V_m \setminus R$. Let $B \in \mathcal{B}$ and let
$z_B$ be the common vertex of the road and the branch $B$. Among all vertices in $B\cap(V_m\setminus R)$ choose $x_B\in B\cap (V_m\setminus R)$ that minimizes $|x_B-z_B|_\infty$.

We claim that $|x_B-z_B|_{\infty}\leq 2r^m$. To prove the claim, note first that if $|z_B-y|_{\infty}>r^m$ for any vertex $y\in B\cap V_N$, then $z_B$ and $y$ belong to two different sets $K_w$, $K_u$, respectively, with $w,u\in A^{*}(r^m)$. By design of $T_N$ and the branch-preserving property, we have that $V_N\cap K_w \subset B$, because the minimal connected subgraph containing those vertices contains no other vertices. Because one of those vertices belongs to $V_m$, we get the claim.

By the claim above and the assumption $|a-z_B|_{\infty}\leq (4/r)r^m$,
we obtain $|a-x_B|_{\infty} \leq (2+4r^{-1})r^m$ for all $B\in\mathcal{B}$.
By Lemma \ref{lem:Remmes1}(4),
the balls $B(x_B,(2r)r^{m})$ are mutually disjoint.
Since $2r<1$, we have $B(x_B,(2r)r^{m}) \subset B_0 := B(a,(3+4r^{-1})r^m)$ for all $B\in\mathcal{B}$. Applying \eqref{e:sreg} twice,
\begin{align*}
(3+4r^{-1})^sr^{ms} \geq C_2^{-1} \mathcal{H}^s(K\cap B_0) &\geq C_2^{-1}\sum_{\substack{B\in \mathcal{B} }}\mathcal{H}^s(K \cap B(x_B,(2r)r^m))\\
&\geq C_1 C_2^{-1} \card(\mathcal{B})(2r)^sr^{ms}
\end{align*}
and we obtain that $\card(\mathcal{B}) \leq C_1^{-1} C_2(3+4r^{-1})^s(2r)^{-s} \lesssim_{L_1,s,\tau,C_1,C_2} 1$.

(2) The proof is similar to that of (1). If $z_i,z_{i+1}$ are as in (2), then
\[ |z_i-z_{i+1}|_{\infty} \geq (2r) r^m > r^{m+1},\]
so there exist distinct $w_i,w_{i+1}\in A^*(r^{m+1})$ (if $m+1 \leq N$) or distinct $w_i,w_{i+1}\in A^*(r^{N})$ (if $m+1 > N$) such that $z_i\in K_{w_i}$ and $z_{i+1}\in K_{w_{i+1}}$.
Because $T_N$ is a tree, it follows that if $i,j\in\{1,\dots,l\}$ with $i\neq j$, then $w_i \neq j$. Therefore, all $z_1,\dots,z_l$ belong to different sets $K_{w_1},\dots,K_{w_l}$. Now we can use \eqref{e:sreg} and work as in (1) to obtain an upper bound for $l$.

(3) Set $C'' = 2C+C'+1$ and set $t' = \lceil \log_{r}\left(2r/C''\right)\rceil$. Because $|a-b|_{\infty}\geq (2r) r^m \geq C'' r^{m+t'}$, the road $R_N(a,b)$ contains at least $C''$ elements of $V_{m+t'}$. Since $\mathcal{F}$ has branching (recall \eqref{e:branching}), there exist at least $C''$ branches of $T_{N,m+t'}$ with respect to $R_{N}(a,b)$. By the branch-preserving property, for each such branch, there exists $w\in A^*(r^{m+t'+1})$ such that the said branch contains all vertices in $V_N\cap K_w$. Thus, we may take $t = \lceil \log_{r}\left(2r/C''\right)\rceil+1$, which ultimately depends at most on $L_1,\tau,s,C_1,C_2$, and (3) holds.
\end{proof}

With Remes' branching lemma (Lemma \ref{lem:Remes}) in hand, we devote the remainder of this section to a proof of Lemma \ref{prelim}. Throughout what follows, we let $t\simeq_{L_1,\tau,s,C_1,C_2}1$ denote the integer given by Lemma \ref{lem:Remes}(3). Instead of proving (P1)--(P7), it is enough to prove (P1)--(P5), (P7), and the following property:
\begin{enumerate}
\item[\emph{(P$6'$)}] \emph{For $1 \leq m\leq N-t$ and $I\in\mathscr{E}_m$, there exists $w\in A^*(r^{m+t})$ such that $f_m(I)$ traces the vertices of $K_w\cap V_N$.}
\end{enumerate} Indeed, let us quickly check that (P6) follows from (P5) and (P$6'$). Fix $m\in\{1,\dots,N\}$ and $I\in \mathscr{E}_m$. Suppose first that $m \leq N-t$. Then, by (P$6'$), $f_m(I)$ is a connected subtree of $T_N$ that contains $V_N\cap K_w$ for some $w\in A^*(r^{m+t})$. Working as in Lemma \ref{6} we get $\card(V_N\cap K_w) \gtrsim_{L_1,s}r^{-(N-m)s}$. So $f_m(I)$ contains at least $cr^{-(N-m)s}$ edges of $T_N$ for some $c\gtrsim_{L_1,s} 1$. Now, by (P5), $f_m(I) \subset f_N(I)$ and (P6) follows when $m \leq N-t$. Suppose otherwise that $m>N-t$. Then $f_m(I)$ contains at least one edge of $T_N$ and $1=r^{-ts}r^{ts}\simeq_{L_1,\tau,s,C_1,C_2} r^{-ts} \geq r^{-(N-m)s}$ and (P6) follows when $t>N-m$.

The construction of the intervals $\mathscr{E}_m$ and the maps $f_m$ satisfying (P1)--(P5) and (P$6'$) is in an inductive manner. We verify (P7) after the construction of the final map $f_N$.

\subsubsection{Initial step.}\label{sec:step1}
Define a collection of nondegenerate closed intervals $\mathscr{E}_1$ as well as auxilliary map $g_1: [0,1] \to T_{N,1}$ so that the following properties hold.
\begin{enumerate}
\item The intervals in $\mathscr{E}_1$ have mutually disjoint interiors and $\bigcup\mathscr{E}_1 = [0,1]$.
\item The map $g_1$ is a 2-to-1 piecewise linear tour of edges of $T_{N,1}$.
\item For each $I \in \mathscr{E}_1$, $g_1$ maps the endpoints of $I$ onto two vertices in $V_1$ and maps $I$ piecewise linearly onto the road  that joins the two vertices in $T_{N,1}$.\end{enumerate}

If $N-t<1$, we simply set $f_1=g_1$ and proceed to the inductive step. Otherwise, $1\leq N-t$ and to define $f_1$, we modify  the map $g_1$ on each interval in $\mathscr{E}_1$ by inserting branches. Let $\{I_1,\dots,I_{n}\}$ be an enumeration of $\mathscr{E}_1$. Let $C$ be as in Lemma \ref{lem:Remes}(1).

\begin{lem}\label{lem:branch11}
Let $I_1 = [x,y]$, $a=g_1(x)$ and $b=g_1(y)$. Let $\{B_1,\dots,B_p\}$ be the branches of $T_N$ with respect to the road $R_N(a,b)$ that contain a set $K_w\cap V_N$ for some $w\in A^*(r^{t+1})$. There exist at most $C$ indices $j\in \{1,\dots,p\}$, for which $B_j$ has parts that are traced by $g_1$.
\end{lem}

\begin{proof}
If $B_j$ is a branch as in the assumption of the lemma, then $B_j$ contains a point in $V_1$. However, by Lemma \ref{lem:Remes}(1), we know that no more than $C$ such branches exist.
\end{proof}

Writing $I_1=[x,y]$, since $|g_1(x)-g_1(y)|_{\infty}> (2r) r$ and $|g_1(x)-z|_{\infty}\leq (1/r)r$ for every vertex $z$ of $R_N(g_1(x),g_1(y))$ in $T_N$, we can invoke Lemma \ref{lem:Remes}(3). Thus, we can find a branch $B$ of $R_N(g_1(x),g_1(y))$ with respect to $T_N$ that contains all vertices of $V_N\cap K_w$ for some $w\in A^*(r^t)$ such that no part of it is traced by $g_1$. We define $f_1|I_1$ so that the following properties are satisfied.
\begin{enumerate}
\item The map $f_1|I_1$ is piecewise linear and traces all the edges of $B \cup g_1(I_1) \subset T_{N}$. (Necessarily, every edge of $B$ is traced exactly twice, once in each direction.)
\item We have $f_1|\partial I_1 = g_1|\partial I_1$.
\end{enumerate}

Suppose that we have defined $f_1$ on $I_1,\dots,I_i$. To define $f_1|I_{i+1}$, we first verify the following analogue of Lemma \ref{lem:branch11}.

\begin{lem}\label{lem:branch1i}
Write $I_{i+1} = [x,y]$, $a=g_1(x)$ and $b=g_1(y)$. Let $\{B_1,\dots,B_p\}$ be the branches of $T_N$ with respect to the road $R_N(a,b)$ that contain a set $K_w\cap V_N$ for some $w\in A^*(r^{t+1})$. There exist at most $2C+1$ indices $j\in \{1,\dots,p\}$ for which $B_j$ has been traced by $f_1|I_1\cup\cdots\cup I_{i}$.
\end{lem}

\begin{proof}
There are two cases in which a branch $B_j$ has been traced by $f_1|I_1\cup\cdots\cup I_i$.
The first case occurs when part of $B_j$ is already traced by $g_1$ (and hence by $f_1|I_1\cup\cdots\cup I_i$). As in Lemma \ref{lem:branch11}, at most $C$ such branches $B_j$ exist.
The second case occurs when we are traveling on the road $R_N(a,b)$ backwards. More specifically, the second case occurs when there exists $i_1\in \{1,\dots,i\}$ such that there is a part of $g_1(I_1)$ lying on $R_N(a,b)$ and part of $B_j$ is being traced by $f_1|I_{i_1}$. In this situation, there are two possible subcases:
\begin{enumerate}
\item the right endpoint of $I_{i_1}$ is mapped by $g_{1}$ into one of the branches of $T_{N,1}$ with respect to $R_N(a,b)$ and by Lemma \ref{lem:Remes}(1) at most $C$ such branches exist; and,
\item $f_{1}|I_{i_1}$ contains $a$ and since $f_1$ is essentially 2-1, at most one such interval exists.
\end{enumerate}
In total, there exist at most $2C+1$ indices $j\in \{1,\dots,p\}$ for which $B_j$ has been traced by $f_1|I_1\cup\cdots\cup I_i$.
\end{proof}

For $I_{i+1}$, we now work exactly as with $I_1$, but we choose a branch $B_j$ that has no edge being traced by $f_1|I_1\cup\cdots\cup I_i$. We can do so because by Lemma \ref{lem:Remes}(3), there exist at least $2C+2$ branches of $T_N$ with respect to the road $R_N(a,b)$ that contain a set $K_w\cap V_N$ for some $w\in A^*(r^{t+1})$. Modifying $g_1$ on each $I_i$ completes the definition of $f_1$.

Properties (P1), (P2), (P3) follow by design of $f_1$ and $\mathscr{E}_1$. For property (P4), given $I = [a,b]\in\mathscr{E}_1$ we have that $f_1(a), f_1(b) \in V_1$ and by Lemma \ref{lem:Remmes1}(4), $|f_1(a)-f_1(b)|_{\infty} \geq (4r)r$. On the other hand, since $\diam{K}=1$ and $f_1([0,1])\subset K$, we trivially have $|f_1(x)-f_1(y)| \leq (4/r)r$ which settles (P4). Property (P5) is vacuous in the initial step (as $f_2$ has not yet been defined). Finally, property (P$6'$) holds, because when $1\leq N-t$, we used Remes' branching lemma to ensure that each $I\in \mathscr{E}_1$ there exists $w\in A^*(r^{t+1})$ such that $f_1$ traces all vertices of $V_N\cap K_w$.

\subsubsection{Inductive step}\label{sec:Remesinduction} Suppose that for some $1\leq m\leq N-1$ we have defined $f_m$ and $\mathscr{E}_m$ so that properties (P1)--(P5) and (P$6'$) hold.

We start by defining an auxiliary map $g_{m+1}$ that visits the image of $f_m$ and $T_{N,m+1}$. In particular, define $g_{m+1}: [0,1] \to T_N$ and an auxiliary collection of intervals $\mathscr{B}_{m+1}$ of nondegenerate closed intervals in $[0,1]$ so that the following properties hold.
\begin{enumerate}
\item The intervals in $\mathscr{B}_{m+1}$ have mutually disjoint interiors and collectively $\bigcup \mathscr{B}_{m+1} = [0,1]$. Moreover, for any $I \in \mathscr{B}_{m+1}$ there exists unique $J\in\mathscr{E}_m$ such that $J\subseteq I$.
\item The map $g_{m+1}$ is a 2-to-1 piecewise linear tour of edges of $T_N$ in $f_m([0,1]) \cup T_{N,m+1}$. For any $I\in\mathscr{B}_{m+1}$, $g_{m+1}|I$ maps $I$ linearly onto an edge of $T_N$ in $f_m([0,1])\cup T_{N,m+1}$.
\item For each $I \in \mathscr{E}_m$, we have $g_{m+1}|\partial I=f_{m}|\partial I$ and $f_{m}(I) \subset g_{m+1}(I)$.
\end{enumerate}
Note that if $T_{N,m+1} \subset f_m([0,1])$ we can choose $g_{m+1} = f_m$.

To define $\mathscr{E}_{m+1}$, we will first identify the endpoints of its intervals. Towards this goal, let $W_{m+1}$ denote the set of endpoints of the intervals in $\mathscr{B}_{m+1}$ and let $P_{m}$ denote the set of endpoints of the intervals in $\mathscr{E}_m$. By definition of $\mathscr{B}_{m+1}$, we have $P_m \subset W_{m+1}$.

\begin{lem}\label{rem:Ppoints}
There exists a maximal set $P_{m+1}$ contained in $W_{m+1}$ with $P_{m+1}\supset P_m$ such that for any consecutive points $x,y \in P_{m+1}$,
\begin{enumerate}
\item $|g_{m+1}(x)-g_{m+1}(y)|_{\infty} \geq (2r)r^{m+1}$, and
\item if $z \in [x,y]$, then $|g_{m+1}(x)-g_{m+1}(z)|_{\infty} \leq (4/r)r^{m+1}$.
\end{enumerate}
\end{lem}

\begin{proof}
We start by making a simple remark. By design of $\mathscr{B}_{m+1}$, for any two consecutive points $x,y \in W_{m+1}$, there exists $w,u\in A^*(r^{N})$ such that $g_{m+1}(x)\in K_w$, $g_{m+1}(y)\in K_u$ and $K_{w}\cap K_u \neq \emptyset$. Hence \begin{equation}\label{eq:Wpoints}
|g_{m+1}(x)-g_{m+1}(y)|_{\infty} \leq 2r^{N}.
\end{equation}
To prove the lemma, it suffices (as $W_{m+1}$ is finite) to construct a set $P_{m+1}'$ such that $P_m \subset P_{m+1}' \subset W_{m+1}$ and $P_{m+1}'$ satisfies the conclusions of the lemma. The definition of $P_{m+1}'$ will be in an inductive manner.
Set $P_{m+1}^{(1)}  = P_m$. By the inductive hypothesis (P4), we have that $|g_{m+1}(x)-g_{m+1}(y)|_{\infty} \geq (2r) r^{m+1}$ for any two consecutive points $x,y \in P_{m+1}^{(1)}$. Assume now that for some $i\in\N$ we have defined $P_{m+1}^{(i)}$ so that $|g_{m+1}(x)-g_{m+1}(y)|_{\infty} \geq (2r) r^{m+1}$ for any two consecutive points $x,y \in P_{m+1}^{(i)}$. To define the next set $P_{m+1}^{(i)}$, we consider two alternatives.

Suppose first that for any two consecutive points $x,y\in P_{m+1}^{(i)}$ with $x<y$ and for any $z\in W_{m+1}\cap [x,y]$, we have $|g_{m+1}(x)-g_{m+1}(z)|_{\infty} \leq (4/r)r^{m+1}$. In this case, we set $P_{m+1}^{(i+1)} := P_{m+1}^{(i)}$.

Suppose now that there exist consecutive $x,y \in P_{m+1}^{(i)}$ with $x<y$ for which the previous situation fails. We claim that there exists $z \in W_{m+1}\cap [x,y]$ such that
\begin{equation}\label{eq:Ppoints}
\max\{ |g_{m+1}(x)-g_{m+1}(z)|_{\infty} , |g_{m+1}(y)-g_{m+1}(z)|_{\infty}\} \geq r^{m+1}.
\end{equation}
To prove (\ref{eq:Ppoints}), assume first that $|g_{m+1}(x)-g_{m+1}(y)|_{\infty}\} \geq 4 r^{m+1}$. Since $g_m([x,y])$ is connected, there exists $x\in W_{m+1}\cap [x,y]$ such that $g_{m+1}(z)$ is not contained in $B(x,r^{m+1})\cup B(y,r^{m+1})$ and (\ref{eq:Ppoints}) holds. Assume now that $|g_{m+1}(x)-g_{m+1}(y)|_{\infty} < 4 r^{m+1}$ and let $z\in W_{m+1}\cap [x,y]$ be such that $|g_{m+1}(x)-g_{m+1}(z)|_{\infty} > (4/r)r^{m+1}$. Since $r< 1/4$,
\begin{align*}
|g_{m+1}(y)-g_{m+1}(z)|_{\infty} &\geq |g_{m+1}(x)-g_{m+1}(z)|_{\infty} - |g_{m+1}(x)-g_{m+1}(y)|_{\infty} \\
&> (4/r)r^{m+1} - 4 r^{m+1} > 4r^{m+1}.
\end{align*}
Having proved (\ref{eq:Ppoints}), we set $P_{m+1}^{(i+1)} := P_{m+1}^{(i)}\cup\{z\}$.

In view of \eqref{eq:Wpoints} and finiteness of the set $W_{m+1}$, there exists a minimal $n\in\N$ with $P_{m+1}^{(n+1)} = P_{m+1}^{(n)}$. Set $P_{m+1}' := P_{m+1}^{(n)}$. It is straight forward to see using induction that the set $P_{m+1}'$ satisfies the conclusions of the lemma.
\end{proof}

Define $\mathscr{E}_{m+1}$ to be the maximal collection of nondegenerate closed intervals in $[0,1]$ whose endpoints are consecutive points in the set $P_{m+1}$. If $m+1>N-t$, set $f_{m+1}:=g_{m+1}$. Otherwise, $m+1\leq N-t$ and to define $f_{m+1}$, we modify $g_{m+1}$ on each $I\in\mathscr{E}_{m+1}$ like we did in the initial step.

Assume $m+1\leq N-t$ and let $\{I_1,\dots,I_{q}\}$ be an enumeration of $\mathscr{E}_{m+1}$. We start with $I_1$. If $g_{m+1}(I_{1})$ traces a branch of $T_N$ with respect to $R_N(a,b)$ that contains all vertices of $V_N\cap K_w$ for some $w\in A^*(r^{m+t+1})$, then we set $f_{m+1}|I_{1} = g_{m+1}|I_{1}$. Suppose now that $g_{m+1}(I_{1})$ does not trace such a branch.

\begin{lem}[cf.~Lemma {\ref{lem:branch11}}] \label{lem:branchm1}
Let $I_1 = [x,y]$, $a=g_{m+1}(x)$ and $b=g_{m+1}(y)$. Let $\{B_1,\dots,B_p\}$ denote the branches of $T_N$ with respect to the road $R_N(a,b)$ that contain a set $K_w\cap V_N$ for some $w\in A^*(r^{m+t+1})$. Then there exist at most $C$ indices $j\in \{1,\dots,p\}$ for which $B_j$ has parts that are traced by $g_{m+1}$.
\end{lem}

\begin{proof}
The branches of $R_N(a,b)$ with respect to $g_{m+1}([0,1])$ that are not in $f_m([0,1])$ are branches that contain points in $V_{m+1}$. Therefore, by Lemma \ref{lem:Remes}(1), there are at most $C$ of them.
\end{proof}

Since $|a-b|_{\infty}> (2r) r^{m+1}$ and $|a-z|_{\infty}\leq (4/r)r^{m+1}$ for every vertex $z$ of $R_N(a,b)$ in $T_N$, we can invoke Lemma \ref{lem:Remes}(3). In particular, there exist at least $2C+2$ branches of $T_{N}$ with respect to the road $R_N(a,b)$ such that for every branch there exists $w\in A^*(r^{m+t+1})$ such that all vertices of $K_w$ are in that branch. Fix such a branch $B$ and define $f_{m+1}|I_1$ so that the following properties are satisfied.
\begin{enumerate}
\item The map $f_{m+1}|I_1$ is piecewise linear and traces all the edges of $B \cup g_{m+1}(I_1) \subset T_{N}$. In fact, every edge of $B$ is traced exactly twice.  Moreover, for any edge $e$ of $B \cup g_{m+1}(I_1)$ there exists $J \subset I_1$ such that $f_{m+1}|I_1$ maps $J$ linearly onto $e$.
\item We have $f_{m+1}|I_1(x) = g_{m+1}(x)$ and $f_{m+1}|I_1(y) = g_{m+1}(y)$.
\end{enumerate}

Suppose that we have defined $f_{m+1}$ on $I_1,\dots,I_i$. Write $I_{i+1}=[x,y]$, let $a=g_{m+1}(x)$ and let $b=g_{m+1}(y)$. If $g_{m+1}(I_{i+1})$ traces a branch of $T_N$ with respect to $R_N(a,b)$ that contains all vertices of $V_N\cap K_w$ for some $w\in A^*(r^{m+t+1})$, then we set $f_{m+1}|I_{i+1} = g_{m+1}|I_{i+1}$. Suppose now that $g_{m+1}(I_{i+1})$ does not trace such a branch.

\begin{lem}[cf.~Lemma \ref{lem:branch1i}] \label{lem:branch2} Let $\{B_1,\dots,B_p\}$ be the branches of $T_N$ with respect to the road $R_N(a,b)$ that contain a set $K_w\cap V_N$ for some $w\in A^*(r^{t+m+1})$. There exist at most $2C+C'+1$ indices $j\in \{1,\dots,p\}$ for which $B_j$ has been traced by $f_{m+1}|I_1\cup\cdots\cup I_i$.
\end{lem}

\begin{proof}
There are two cases in which a branch $B_j$ has been traced by $f_1|I_1\cup\cdots\cup I_i$.
The first case is when part of $B_j$ is already traced by by $g_{m+1}$ (and hence $f_{m+1}|I_1\cup\cdots I_i$). As in Lemma \ref{lem:branchm1}, at most $C$ such branches exist.

The second case is when we are traveling on the road $R_N(a,b)$ backwards. Specifically, this case occurs when there exists $i_1\in \{1,\dots,i\}$ such that there is a part of $g_{m+1}(I_1)$ lying on $R_N(a,b)$ and part of $B_j$ is being traced by $f_{m+1}|I_{i_1}$. There are three possible subcases:
\begin{enumerate}
\item the right endpoint of $I_{i_1}$ is mapped by $g_{m+1}$ into one of the branches of $T_{N,1}$ with respect to $R_N(a,b)$ and by Lemma \ref{lem:Remes}(1) at most $C$ such branches exist;
\item the right endpoint of $I_{i_1}$ is mapped onto the road $R_N(a,b)$ and by Lemma \ref{lem:Remes}(2) at most $C'$ such points exist; and,
\item $f_{m+1}|I_{i_1}$ contains $a$, and since $f_{m+1}$ is essentially 2-to-1, at most one such interval exists.
\end{enumerate}
In total, there exist at most $2C+C'+1$ indices $j\in \{1,\dots,p\}$, for which $B_j$ has been traced by $f_{m+1}|I_1\cup\cdots\cup I_i$.
\end{proof}

For $I_{i+1}$ we work exactly as with $I_1$, but we choose a branch $B$ that has not been traced by $f_{m+1}|I_1\cup\cdots\cup I_i$. We can do so because by Lemma \ref{lem:Remes}(3), there exist at least $2C+C'+2$ such branches. Modifying $g_{m+1}$ on each $I_i$ completes the definition of $f_{m+1}$.

\subsubsection{Properties (P1)--(P5) and (P$6'$) for the inductive step}

We complete the inductive step by proving properties (P1)--(P5) and (P$6'$). Properties (P1), (P2), (P3) and (P$6'$) follow immediately by design of $\mathscr{E}_{m+1}$ and $f_{m+1}$.

For (P4), fix $I = [a,b] \in\mathscr{E}_{m+1}$. The first claim of (P4) follows by Lemma \ref{rem:Ppoints} and the fact that $f_{m+1}|\partial I = g_{m+1}|\partial I$. For the second claim, let $x\in [a,b]$. If $f_{m+1}(x) \in g_{m+1}([a,b])$ (which e.g.~ always happens when $m+1>N-t$), then
\[ |f_{m+1}(x) - f_{m+1}(a)|_{\infty} \leq (4/r)r^{m+1}\]
by Lemma \ref{rem:Ppoints}. If $f_{m+1}(x) \not\in g_{m+1}([a,b])$ (which can only happen when $m+1\leq N-t$), then $f_{m+1}(x)$ is contained in a branch $B$ of $T_N$ with respect to $R_N(f_{m+1}(a),f_{m+1}(b))$. Thus, $\diam{B} \leq r^{m}$, and if $z\in [a,b]$ with $f_{m+1}(z) \in B\cap g_{m+1}([a,b])$, then
\begin{align*}
|f_{m+1}(x) - f_{m+1}(a)|_{\infty} &\leq |f_{m+1}(z) - f_{m+1}(a)|_{\infty} + |f_{m+1}(x) - f_{m+1}(z)|_{\infty}\\
&\leq (4/r)r^{m+1} + \diam{B}\\
&\leq (4/r)r^{m+1} + (r^{m+2} + \cdots + r^{N} )\leq (5/r)r^{m+1}.
\end{align*}

For (P5), fix $I \in \mathscr{E}_m$. By design of $f_{m+1}$ and $g_{m+1}$, we have $f_m(I) \subset g_{m+1}(I)$ and $g_{m+1}(I) \subset f_{m+1}(I)$. Thus, $f_m(I) \subset f_{m+1}(I)$. Let $x$ be an endpoint of $I$. On one hand, $g_{m+1}(x) = f_m(x)$. On the other hand, there exists $J \in \mathscr{E}_{m+1}$ with $x$ as its endpoint, and by construction, $f_{m+1}|\partial J = g_{m+1}|\partial J$. Therefore, $f_{m+1}|\partial I = f_m|\partial I$.

\subsubsection{Property (P7)} To prove (P7), suppose that $m+1 = N$. Since $m+1 = N > N-t$, the map $f_{m+1} = g_{m+1}$. By definition, $g_{m+1}([0,1])$ contains $T_{N,N} = T_N$, so $\widetilde{T}_{N,m+1} = \widetilde{T}_{N,N} = f_{m+1}([0,1]) = T_N$. Moreover, since $W_{m+1}$ satisfies both conclusions of Lemma \ref{rem:Ppoints}, $W_{m+1} = P_{m+1}$. Hence $\mathscr{E}_{m+1} = \mathscr{B}_{m+1}$. Thus, since every interval from $\mathscr{B}_{m+1}$ is mapped by $g_{m+1}$ linearly onto an edge of $T_N$, every interval from $\mathscr{E}_{m+1}$ is mapped by $f_{m+1}$ linearly onto an edge of $T_N$.

With persistence, we have completed the proof of Lemma \ref{prelim}.

\section{Bedford-McMullen carpets and self-affine sponges}\label{sec:carpets}
Self-affine carpets were introduced and studied independently by Bedford \cite{Bedford} and McMullen \cite{McMullen}. Fix integers $2 \leq n_1\leq n_2$. For each pair of indices $i\in\{1,\dots,n_1\}$ and $j\in\{1,\dots,n_2\}$, let $\phi_{i,j} : \R^2 \to \R^2$ be the affine contraction given by
\[ \phi_{i,j}(x,y) = (n_1^{-1}(i-1+x) , n_2^{-1}(j-1+y))\quad\text{with }\Lip\phi_{i,j}=n_1^{-1}.\]
For each nonempty set $A\subset\{1,\dots,n_1\}\times\{1,\dots,n_2\}$, we associate the iterated function system $\mathcal{F}_A=\{\phi_{i,j} : (i,j)\in A\}$ over $\R^2$ and let $\mathcal{S}_A$ denote the attractor of $\mathcal{F}_A$, called a \emph{Bedford-McMullen carpet}. In general, we have $\mathcal{S}_A \subset [0,1]^2$.

The following proposition serves as a brief overview of how the similarity dimension of $\mathcal{F}_A$ compares to the Hausdorff, Minkowski, and Assouad dimensions of the carpet $\mathcal{S}_A$; for definitions of these dimensions, we refer the reader to \cite{McMullen} and \cite{Mackay}.

\begin{prop}\label{prop:BMdimensions}
Let $2 \leq n_1\leq n_2$ and $A$ be as above. For all $i\in\{1,\dots,n_1\}$, define
\[ t_i := \card\{j : (i,j)\in A\}.\] Also define $t:=\max_i t_i$ and $r:=\card\{i : t_i \neq 0\}$.
\begin{enumerate}
\item The similarity dimension is
\[ \sdim(\mathcal{F}_A) = \log_{n_1}\left(\sum_{i=1}^{n_1}t_i\right) = \log_{n_1}(\card{A}).\]
\item \cite{McMullen} The Hausdorff dimension is
\[ \dim_H(\mathcal{S}_A) =  \log_{n_1}\left(\sum_{i=1}^{n_1}t_i^{\log_{n_2}{n_1}}\right).\]
\item \cite{McMullen} The Minkowski dimension is
\[ \dim_M(\mathcal{S}_A) = \log_{n_1}{r} + \log_{n_2}\left(r^{-1}\sum_{i=1}^{n_1}t_i\right) = \log_{n_1}{r}+ \log_{n_2}(r^{-1}\card{A}).\]
\item \cite{Mackay} If $n_1<n_2$, then the Assouad dimension is
\[ \dim_A(\mathcal{S}_A) = \log_{n_1}{r} + \log_{n_2}{t} .\]
\end{enumerate}
\end{prop}

It is easy to see that for every Bedford-McMullen carpet, \begin{equation}\dim_H(\mathcal{S}_A) \leq \dim_M(\mathcal{S}_A) \leq \min\{ \dim_A(\mathcal{S}_A), \sdim(\mathcal{F}_A)\}.\end{equation} However, there is no universal comparison between the Assouad and similarity dimensions. In fact, there are examples of self-affine carpets showing that $\dim_{A}(\mathcal{S}_A) < \sdim(\mathcal{F}_A)$, $\dim_{A}(\mathcal{S}_A)=\sdim(\mathcal{F}_A)$, and $\dim_{A}(\mathcal{S}_A) > \sdim(\mathcal{F}_A)$ are each possible. We emphasize that the similarity dimension of a self-affine carpet can exceed 2 (see Figure \ref{fig:mc23}).

\subsection{H\"older parameterization of connected Bedford-McMullen carpets with sharp exponents}
For each index $i\in\{1,\dots,n_1\}$, define $A_i := \{i\}\times \{1,\dots,n_2\}$ and $A_0 :=  \bigcup_{i=1}^{n_1}A_i$. Note that the carpet $\mathcal{S}_{A_0}=[0,1]^2$, and for each $i\in\{1,\dots,n_1\}$, the carpet $\mathcal{S}_{A_i}$ is the vertical line segment $\{(i-1)/(n_1-1)\}\times[0,1]$ (see Figure \ref{fig:carpets}).

\begin{figure}[th]\begin{center}
\includegraphics[width=.8\textwidth]{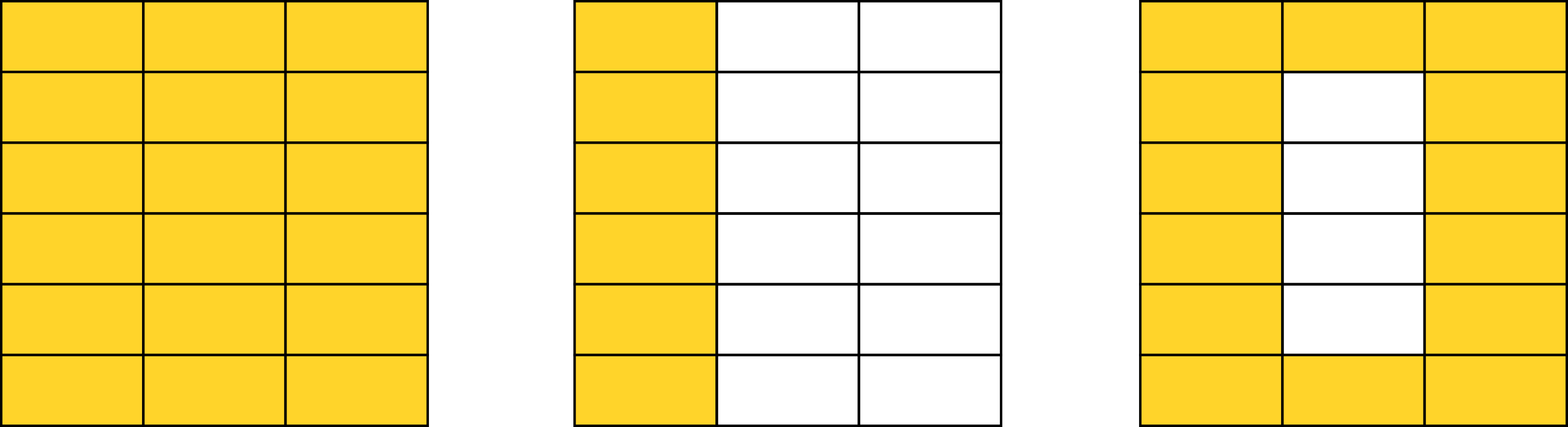}\end{center}
\caption{First iteration of Bedford-McMullen carpets with generators $A$. On the left, $A=A_0$ (the square). In the middle, $A=A_1$ (a vertical line). On the right, $A= \{(1,1),\dots,(1,6),(2,1),(2,6),(3,1),\dots,(3,6)\}.$}\label{fig:carpets}
\end{figure}

Our goal in this section is to establish the following statement, which encapsulates Theorem \ref{thm:carpets2} from the introduction.

\begin{thm}[H\"older parameterization] \label{thm:carpets}
Let $2 \leq n_1\leq n_2$ be integers and let $A$ be as above. If $\mathcal{S}_A$ is connected, then there exists a surjective $(1/\a)$-H\"older map $F:[0,1]\to \mathcal{S}_{A}$ with
\[
\a=
\begin{cases}
\text{arbitrary}, &\text{ if $\card(A)=1$};\\
1, &\text{ if $A=A_i$ for some $i\in\{1,\dots,n_1\}$};\\
2, &\text{ if $A=A_0$};\\
\sdim(\mathcal{F}_A), &\text{ otherwise}.
\end{cases}
\]
Furthermore, the exponent $1/\a$ is sharp.
\end{thm}

Note that the conclusion of Theorem \ref{thm:carpets} is trivial in the case that $A\in\{A_0,\dots,A_{n_1}\}$ or in the case that $\card{A}=1$. Below we give a proof of the sharpness of the exponent $\a$, and in \textsection\ref{sec:snow} we show why such a surjection exists.

\begin{lem}\label{lem:specialbox}
If $\mathcal{S}_A$ is connected and $A\not\in \{A_0,\dots,A_{n_1}\}$, then  there exists a pair of indices $(i,j) \in A$ such that $j<n_2$ and $(i,j+1) \not\in A$ or such that $j>1$ and $(i,j-1) \not\in A$.
\end{lem}

\begin{proof} To establish the contrapositive, suppose that the conclusion of the lemma fails. Then $A = B\times\{1,\dots,n_2\}$ for some nonempty set $B\subset\{1,\dots,n_1\}$. If $\card(B)=1$, then $A=A_i$ for some $1\leq i\leq n_1$. If $1<\card(B)<n_1$, then the carpet $\mathcal{S}_A$ is disconnected. Finally, if $\card(B)=n_1$, then $A=A_0$.
\end{proof}

\begin{lem}\label{lem:connectedcarpet}
Suppose that $\mathcal{S}_A$ is connected, $\card{A}\geq 2$, and $A\not\in \{A_1,\dots,A_{n_1}\}$. Then the ``first iteration" $\bigcup_{(i,j)\in A}\phi_{i,j}([0,1]^2)$ is a connected set that intersects both the left and the right edge of $[0,1]^2$.
\end{lem}

\begin{proof} If $\card{A}\geq 2$, $A\not\in\{A_1,\dots,A_{n_1}\}$, and the ``first iteration" $\bigcup_{(i,j)\in A}\phi_{i,j}([0,1]^2)$ does not touch the left or right edge, then the ``second iteration" $\bigcup_{(i,j),(k,l)\in A} \phi_{i,j}\circ \phi_{j,k}([0,1]^2)$ is disconnected. We leave the details as a useful exercise for the reader. It may help to visualize the diagrams in Figures \ref{fig:mc23} or \ref{fig:carpets}.\end{proof}

\begin{cor}\label{cor:connectedcarpets}
Suppose that $\mathcal{S}_A$ is connected, $\card A\geq 2$, and $A\not\in \{A_1,\dots,A_{n_1}\}$. Then $\mathcal{S}_A$ intersects both left and right edge of $[0,1]^2$.
\end{cor}

We are ready to prove Theorem \ref{thm:carpets}.

\begin{proof}[{Proof of Theorem \ref{thm:carpets}}] With the conclusion being straightforward otherwise, let us assume that $\mathcal{S}_A$ is a connected Bedford-McMullen carpet with $\card{A}\geq 2$ and $A\not\in\{A_0,\dots,A_{n_1}\}$. Let $s=\sdim{\mathcal{F}_A}$. We defer the proof of existence of a $(1/s)$-H\"older parameterization of $\mathcal{S}_A$ to \textsection\ref{sec:snow}, where we prove existence of H\"older parameterizations for self-affine sponges in $\R^N$ (see Corollary \ref{cor:sponge}). It remains to prove the sharpness of the exponent $1/s$.

Set $k=\card{A}$ and suppose that $f:[0,1] \to \mathcal{S}_{A}$ is a $(1/\a)$-H\"older surjection for some exponent $\alpha>0$. Since $\mathcal{S}_A$ has positive diameter, the H\"older constant $H:=\Hold_{1/\a} f>0$. By Proposition \ref{prop:BMdimensions}, $\sdim\mathcal{F}_A=\log_{n_1}(k)$. Thus, we must show that $\a\geq \log_{n_1}{k}$.

Fix $m\in\N$ and let $A^m$, $A^*$, and $\phi_w$ be defined as in \S\ref{sec:words} relative to the alphabet $\{(i,j):1\leq i\leq n_1,1\leq j\leq n_2\}$.
For each $m\in\N$ and each word $w = (i_1,j_1)\cdots (i_m,j_m)$,
set $S_w = \phi_w([0,1]^2)$. Let $(i_0,j_0)\in A$ be an index given by Lemma \ref{lem:specialbox}, i.e.~an address in the first iterate such that the rectangle either immediately above or below is omitted from the carpet. Without loss of generality, we assume that $j_0<n_2$ and $(i_0,j_0+1) \not\in A$ (there is no rectangle below $(i_0,j_0)$). Moreover, we assume that
\[ j_0 = \min\{j : (i_0,j) \in A \text{ and }(i_0,j+1)\not\in A\}.\]
For each word $w\in A^m$, we now define a ``column of rectangles" $\tilde{S}_w$, as follows.

\emph{Case 1.} If $S_w$ intersects the bottom edge of $[0,1]^2$, then set $\tilde{S}_w = \bigcup_{j=0}^{j_0} S_{w(i_0,j)}$.

\emph{Case 2.} Suppose that $S_w$ does not intersect the bottom edge of $[0,1]^2$. Let $u=(i_1,j_1)\cdots(i_m,j_m)$ with
\[(i_1,j_1),\dots,(i_m,j_m) \in \{1,\dots,n_1\}\times\{1,\dots,n_2\}\]
such that the upper edge of $S_u$ is the same as the lower edge of $S_w$. This case is divided into three subcases.

\emph{Case 2.1.} Suppose that $u\not\in A^m$. Then, as in Case 1, set $\tilde{S}_w = \bigcup_{j=0}^{j_0} S_{w(i_0,j)}$.

\emph{Case 2.2.} Suppose that $u\in A$ and $u(i_0,n_2) \not\in A^{m+1}$. Then we set  $\tilde{S}_w = \bigcup_{j=0}^{j_0} S_{w(i_0,j)}$.

\emph{Case 2.3.} Suppose that $u\in A^m$ and $u(i_0,n_2) \in A^{m+1}$. Let $j_1 =\max \{ j : (i_0,j-1)\not\in A \}$. Then we set $\tilde{S}_w =\left ( \bigcup_{j=0}^{j_0} S_{w(i_0,j)}\right) \cup \left(\bigcup_{j=j_1}^{n_2} S_{u(i_0,j)} \right)$.

In each case, $\tilde{S}_w\cap\mathcal{S}_A$ is a connected set that intersects both the left and right edges of $\tilde{S}_w$, but does not intersect the rectangles $S_u$ immediately above and below $\tilde{S}_w$. Moreover, the sets $\tilde{S}_w$ have mutually disjoint interiors. If $\tau_w$ is the line segment joining the midpoints of upper and lower edges of $\tilde{S}_w$, then $\tau_w$ contains a point of $\mathcal{S}_A$, which we denote by $x_w$.

Consequently, there exists $I_w \subset [0,1]$ such that $f(I_w)$ is a curve in $\tilde{S}_w$ joining $x_w$ with one of the left/right edges of  $\tilde{S}_w$. Clearly, the intervals $I_w$ are mutually disjoint and
\[ 1 \geq \sum_{w\in A^m}\diam{I_w} \geq H^{-\a}\sum_{w\in A^m}(\diam{f(I_w)})^{\a} \gtrsim_{H,\a} \sum_{w\in A^m}(2n_1^{m+1})^{-\a} \gtrsim_{n_1,\a} (kn_1^{-\a})^{m}.\]
Since $m$ is arbitrary, $\a\geq \log_{n_1}{k}$.
\end{proof}

\subsection{\!\!Lipschitz lifts~and~H\"older~parameterization~of~connected self-affine~sponges}\label{sec:snow}
Analogues of the Bedford-McMullen carpets in higher dimensional Euclidean spaces are called \emph{self-affine sponges}; for background and further references, see \cite{Kenyon-Peres}, \cite{Das-Simmons}, \cite{Fraser-Howroyd}. To describe a self-affine sponge, let $N\geq 2$ and let $2\leq n_1\leq\cdots\leq n_N$ be integers. For each $n$-tuple $\textbf{i} = (i_1,\dots,i_N) \in \{1,\dots,n_1\}\times\cdots\times\{1,\dots,n_N\}$, we define an affine contraction $\phi_{\textbf{i}} : \R^N \to \R^N$ by
\[ \phi_{\textbf{i}}(x_1,\dots x_N) = (n_1^{-1}(i_1-1+x_1),\dots, n_N^{-1}(i_N-1+x_N))\quad\text{with }\Lip \phi_{\mathbf{i}}=n_1^{-1}.\]
For every nonempty set $A \subset \{1,\dots,n_1\}\times\cdots\times\{1,\dots,n_N\}$, we associate an iterated function system $\mathcal{F}_A = \{\phi_{\textbf{i}}: \textbf{i}\in A\}$ over $\R^N$ and let $\mathcal{S}_A$ denote the attractor of $\mathcal{F}_A$, which we call a self-affine sponge.

Our strategy to parameterize a connected Bedford-McMullen carpet or self-affine sponge is to construct a Lipschitz lift of the set to a self-similar set in a metric space for which  we can invoke Theorem \ref{thm:remes}. Then the H\"older parameterization of the self-similar set descends to a H\"older parameterization of the carpet or sponge.

\begin{lem}[Lipschitz lifts]\label{lem:snow}
Let $N\geq 2$ be an integer, let $2\leq n_1\leq\cdots\leq n_N$ be integers, and let $A$ be a nonempty set as above. There exists a doubling metric $d$ on $\R^N$ such that if $\widetilde{\mathcal{S}}_A$ denotes the attractor of the IFS $\widetilde{\mathcal{F}}_A=\{\phi_\mathbf{i}:\mathbf{i}\in A\}$ over $(\R^N,d)$, then
\begin{enumerate}
\item the identity map $\mathrm{Id}: \widetilde{\mathcal{S}}_A\rightarrow \mathcal{S}_A$ is a $1$-Lipschitz homeomorphism;
\item $\sdim\widetilde{\mathcal{F}}_A=\sdim\mathcal{F}_A=\log_{n_1}(\card A)=:s$, $\widetilde{\mathcal{S}}_A$ is self-similar, and $\Haus^s(\widetilde{\mathcal{S}}_A)>0$.
\end{enumerate}
\end{lem}

\begin{proof}
Consider the product metric $d$ on $\R^N$ given by
\[ d((x_1,\dots,x_N),(x_1',\dots,x_N')) = \left(\sum_{i=1}^N|x_i-x_i'|^{2\log_{n_i}{n_1}}\right)^{1/2}. \]
In other words, $d$ is a metric obtained by ``snowflaking" the Euclidean metric separately in each coordinate.
Note that if $n_1=\cdots=n_N$, then $d$ is the Euclidean metric. It is straightforward to check that $(\R^N,d)$ is a doubling metric space and the identity map $\mathrm{Id}: (\mathcal{S}_A,d) \to \mathcal{S}_A$ is a 1-Lipschitz homeomorphism; e.g.~see Heinonen \cite{Heinonen}.
We now claim that the affine contractions $\phi_{\textbf{i}}$ generating the sponge $\mathcal{S}_A$ become similarities in the metric space $(\R^N,d)$. Indeed, let $\textbf{i} = (i_1,\dots,i_N)\in A$. Then
\begin{align*}
d(\phi_{\textbf{i}}(x_1,\dots,x_N),\phi_{\textbf{i}}(x_1',\dots,x_N')) &= \left(\sum_{i=1}^Nn_i^{-2\log_{n_i}{n_1}}|x_i-x_i'|^{2\log_{n_i}{n_1}}\right)^{1/2} \\
&= n_1^{-1} d((x_1,\dots,x_N),(x_1',\dots,x_N')).
\end{align*}
Since each of the similarities $\phi_{\textbf{i}}$ have scaling factor $n_1^{-1}$, it follows that $$\sdim(\widetilde{\mathcal{F}}_A)=\sdim\mathcal{F}_A=\log_{n_1}(\card A)=:s$$ Finally, $\widetilde{\mathcal{F}}_A$ satisfies the strong open set condition (SOSC) with $U=(0,1)^N$. Therefore, $\Haus^s(\widetilde{\mathcal{S}}_A)>0$ by Theorem \ref{thm:Stella}, since doubling metric spaces are $\beta$-spaces.
\end{proof}

\begin{cor}\label{cor:sponge}
If $\mathcal{S}_A$ is a connected self-affine sponge in $\R^N$, then $\mathcal{S}_A$ is a $(1/s)$-H\"older curve, where $s=\log_{n_1}(\card A)$ is the similarity dimension of $\mathcal{F}_A$.
\end{cor}

\begin{proof} Let $\widetilde{\mathcal{S}}_A$ denote the lift of the sponge $\mathcal{S}_A$ in Euclidean space $\R^N$ to the metric space $(\R^N,d)$ given by Lemma \ref{lem:snow}. By Lemma \ref{lem:snow} (2), the lifted sponge $\widetilde{\mathcal{S}}_A$ is a self-similar set and $\mathcal{H}^{s}(\widetilde{\mathcal{S}}_A) >0$, where $s=\sdim \widetilde{\mathcal{F}}_A=\sdim\mathcal{F}_A=\log_{n_1}(\card A)$. By Remes' theorem in metric spaces (Theorem \ref{thm:remes}), there exists a  $(1/s)$-H\"older surjection $F:[0,1]\to \widetilde{\mathcal{S}}_A$. By Lemma \ref{lem:snow} (1), the identity map $\mathrm{Id}: \widetilde{\mathcal{S}}_A\rightarrow \mathcal{S}_A$ is a Lipschitz homeomorphism. Therefore, the composition $G=[0,1]\rightarrow \mathcal{S}_A$, $G:=\mathrm{Id}\circ F$ is a $(1/s)$-H\"older surjection.
\end{proof}

\bibliography{tsp-refs}
\bibliographystyle{amsbeta}

\end{document}